\documentclass[pt11]{article}
\usepackage{latexsym,amsmath,url,caption2,epsfig}
\usepackage{amsfonts,euscript}
\usepackage{graphicx}
\usepackage{amsmath}
\usepackage{amssymb}
\usepackage{amsfonts}
\usepackage{amsthm}
\usepackage{mathrsfs}
\usepackage[all]{xy}
\usepackage{epstopdf}
\usepackage{subfigure}
\usepackage{rotating}
\usepackage{appendix}
\usepackage[english]{babel}
\usepackage{listings}
\usepackage{color}
\usepackage{tikz}
\usepackage{graphics}
\usepackage{rotfloat}
\usepackage{framed}
\usepackage{titlesec}
\usepackage{amsmath}
\usepackage{amssymb}
\usepackage{amsfonts}
\usepackage{amsthm}
\usepackage{mathrsfs}
\usepackage[numbers]{natbib}
\usepackage{sectsty}

\RequirePackage{amsthm,amsmath}
\RequirePackage[numbers]{natbib}
\RequirePackage[colorlinks,citecolor=blue,urlcolor=blue]{hyperref}

\newcommand{\be}{\begin{equation}}
\newcommand{\ee}{\end{equation}}
\newcommand{\ba}{\begin{aligned}}
\newcommand{\ea}{\end{aligned}}

\newtheoremstyle{mystyle}
  {}
  {}
  {\itshape}
  {}
  {\bfseries}
  {.}
  { }
  {}
\theoremstyle{mystyle}
\newtheorem{definition}{Definition}

\newtheorem{theorem}{Theorem}

\newtheorem{remark}{Remark}

\newtheorem{proposition}{Proposition}

\numberwithin{equation}{section}
\input{tcilatex}
\sectionfont{\centering \mdseries \scshape \large}
\begin{document}

\renewenvironment{proof}[1][\proofname]{{\itshape#1. }}{\qed}

%\begin{frontmatter}
\title{Nonlinear Large Deviations: Beyond the Hypercube}
%\runtitle{Nonlinear Large Deviations: Beyond the Hypercube}
%\thankstext{T1}{Footnote to the title with the ``thankstext'' command.}

%\begin{aug}
\author{Jun Yan}
%\author{\fnms{Jun} \snm{Yan}\ead[label=e1]{junyan65@stanford.edu}},
\maketitle

%\runauthor{Jun Yan}

%\affiliation{Stanford University}

%\address{Jun Yan, Statistics Department, Stanford University, CA, 94305, USA\\
%\printead{e1}\\
%\phantom{E-mail:junyan65@stanford.edu\ }
%\printead*{e2}
%}

%\end{aug}

\begin{abstract} By extending Chatterjee and Dembo \cite{NLD}, we present a framework to calculate large deviations for nonlinear functions
of independent random variables supported on compact sets in Banach spaces. Previous
research on nonlinear large deviations has only focused on random variables
supported on $\{-1,+1\}^{n}$, and accordingly we build
theory for random variables with general distributions, increasing flexibility in applications. As examples, we compute the large deviation
rate functions for monochromatic subgraph counts in edge-colored complete graphs, and for triangle counts in dense random graphs with continuous edge weights. Moreover, we verify the mean field approximation
for a class of vector spin models.
\end{abstract}

%\begin{keyword}[class=MSC]
%\kwd{60F10}
%\kwd{60C05}
%\kwd{05C80}
%\kwd{60K35}
%\end{keyword}

%\begin{keyword}
%\kwd{large deviations}
%\kwd{mean-field}
%\kwd{random graphs}
%\kwd{vector spin models}
%\end{keyword}

%\end{frontmatter}

\section{Introduction}

Large deviations theory for the linear function of i.i.d. random objects has
long been studied, see Dembo and Zeitouni \cite{DZ10} and references
therein. Since the linear function is the simplest class of functions to
analyze and only accounts for a small subset of functions people usually
study, it is of natural interest to explore a corresponding theory for
nonlinear functions. Recently, a nonlinear large deviations framework was
built in Chatterjee and Dembo \cite{NLD}, where the authors deal with the
large deviation principles for nonlinear functions of i.i.d. Bernoulli
random variables. The main theorem in \cite{NLD} gives error bounds of the
mean field approximation of $\log \mathbb{E}_{\mu }[e^{f(X_{1},\ldots
,X_{n})}]$ where $\mu $ is the uniform distribution on $\{-1,+1\}^{n}$. The
error bounds consist of two parts: the complexity terms which involve the
covering number of $\nabla f$, and the smoothness terms which involve the
first two derivatives of $f$. Motivated by \cite{NLD}, Eldan \cite{RE} comes
up with a different nonlinear large deviations framework to deal with
nonlinear functions of i.i.d. random variables supported on $\{-1,+1\}^{n}$.
In \cite{RE}, instead of the covering number of $\nabla f$, a different
notion of complexity called Gaussian width of the discrete gradient of $f$
is introduced, and there $f$ is not required to have the second derivative.
In \cite{NLD} many exciting applications are presented, suggesting the
strong power of the new framework. Using the different method, \cite{RE}
gets stronger results for the examples in \cite{NLD}. However, all of the
examples in \cite{NLD} and \cite{RE}\ concern random variables with
distributions supported on $\{-1,+1\}^{n}$, a small subset of random objects
people usually study in probability theory. Therefore it is natural to
research whether a similar nonlinear large deviations regime works for
random objects with more general distributions, and we can expect it since
the Bernoulli random variable should not be special. Indeed, a framework
similar to \cite{NLD} is used in Basak and Mukherjee \cite{AS} to verify the
universality of the mean field approximation on the Potts model.

In this work, we extend the framework of \cite{NLD} to independent random
variables compactly supported on Banach spaces. Similar to \cite{NLD}, our
main result (Theorem \ref{main-theorem}) gives error bounds for the mean
field approximation of $\log \mathbb{E}_{\mu }[e^{f(X_{1},\ldots ,X_{n})}]$,
while $\mu =\mu _{1}\times \ldots \times \mu _{n}$ could be more general
than \cite{NLD}. Our result has considerable flexibility in applications,
because: (1) $\mu _{i}$'s could be defined on general Banach spaces, and
thus there is no dimension constraint on the supports of $\mu _{i}$'s; (2) $%
\mu _{i}$'s are not required to be discrete; (3) $X_{1},\ldots ,X_{n}$ are
not required to be i.i.d. - only independence is needed. To show this
flexibility we provide examples with high dimensional and continuous random
variables, including an example in which the dimension of the support of $%
\mu _{i}$'s is increasing with $n$; previous methods do not work on these
examples. While we take the same approach as \cite{NLD} in proving our main
result (Theorem \ref{main-theorem}), in \cite{NLD} special calculations for
the product Bernoulli distribution are used, and we find general arguments
for Banach spaces. While our result works for general problems, we propose
that for specific problems the error bounds in Theorem \ref{main-theorem}
could be improved by using the particular structures of the problems. As an
example, we extend the result of \cite{AS} by verifying the mathematical
rigor of the mean field approximation for a larger class of vector spin
models. Note that it will also naturally be of interest to extend the
framework in \cite{RE} for general distributions. However, when proving
theorems for distributions supported on $\{-1,1\}^{n}$, \cite{RE} constructs
a Brownian motion running on $[-1,1]^{n}$, such that whenever a facet of $%
[-1,1]^{n}$ is hit the corresponding coordinate stops moving. In this way
the Brownian motion ends up at $\{-1,1\}^{n}$ uniformly, and one can change
the distribution of the ending point by adding a drift to the Brownian
motion. It is not clear what the corresponding objects should be for general
supports.

\subsection{\textbf{The main result}\label{section-main-result}}

Our goal is to find the leading term of $\log \mathbb{E}_{\mu
}[e^{f(X_{1},\ldots ,X_{n})}]$, for $X_{1},\ldots ,X_{n}$ following a
product measure $\mu $ supported on a compact subset of Banach spaces and $f$
a twice Fr\'{e}chet differentiable functional (see Definition \ref%
{def-differentiable}). As demonstrated in Section \ref{ex-multi-color} and
Section \ref{ex-continuous}, such leading term provides us with the large
deviation rate function. It further plays an important role in statistical
physics, as shown in Section \ref{ex-3}. In Theorem \ref{main-theorem}, we
provide error bounds for the mean field approximation (introduced below) of $%
\log \mathbb{E}_{\mu }[e^{f(X_{1},\ldots ,X_{n})}]$ (\ref%
{eq-mean-field-approximation}), in terms of the covering number of the
gradient $\nabla f$ and the norms of the first two derivatives of $f$. One
should then show on a case by case basis that the error terms are of a
smaller order than the mean field approximation. In Section \ref%
{section-applications} we provide three examples, demonstrating how the
latter task is achieved.

For two probability measures $\xi _{1},\xi _{2}$ on the same space $\Omega $%
, denote by $D(\xi _{1}\parallel \xi _{2})$ the Kullback--Leibler divergence%
\begin{equation*}
D(\xi _{1}\parallel \xi _{2}):=\int_{\Omega }\log \left( \frac{d\xi _{1}}{%
d\xi _{2}}(y)\right) \xi _{1}(dy),
\end{equation*}%
where $\frac{d\xi _{1}}{d\xi _{2}}(\cdot )$ is the Radon--Nikodym
derivative, and we set $D(\xi _{1}\parallel \xi _{2})\equiv \infty $ when
the Radon--Nikodym derivative does not exist. From the Gibbs variational
principle, we have the following identity%
\begin{equation}
\log \mathbb{E}_{\mu }[e^{f(X_{1},\ldots ,X_{n})}]=\max_{\nu \ll \mu
}\left\{ \mathbb{E}_{\nu }[f(X_{1},\ldots ,X_{n})]-D(\nu \parallel \mu
)\right\} .  \label{eq-gibbs}
\end{equation}%
The maximum on the right-hand side of (\ref{eq-gibbs}) is taken over all
measures with $\nu \ll \mu $, which is difficult to analyze. Restricting $%
\nu $ to be a product measure leads to the previously mentioned mean field
approximation:%
\begin{equation}
\log \mathbb{E}_{\mu }[e^{f(X_{1},\ldots ,X_{n})}]\approx \max_{\nu \ll \mu
,\nu =\nu _{1}\times \nu _{2}\times \ldots \times \nu _{n}}\left\{ \mathbb{E}%
_{\nu }[f(X_{1},\ldots ,X_{n})]-D(\nu \parallel \mu )\right\} ,
\label{eq-mean-field-approximation}
\end{equation}%
which is much easier to deal with.

We next introduce some definitions needed for stating our main result. Let $%
[n]:=\{1,\ldots ,n\}$. For each $i\in \lbrack n]$, we consider the
probability space $(V_{i},\mathcal{B}_{i},\mu _{i})$, where $V_{i}$ is a
Banach space (over the field $\mathbb{R}$) equipped with norm $\left\Vert
\cdot \right\Vert _{V_{i}}$, $\mathcal{B}_{i}$ is the Borel $\sigma $%
-algebra generated by $V_{i}$'s open sets, and $\mu _{i}$ is a probability
measure on the measurable space $(V_{i},\mathcal{B}_{i})$. We assume that
for each $i$, there exists a compact convex set $W_{i}\subset $ $V_{i}$ such
that $\mu _{i}(W_{i})=1$. Consider the product probability measure $\mu $
supported on the product space $W$ in $V$ where%
\begin{equation*}
\mu :=\mu _{1}\times \ldots \times \mu _{n}\text{, \ }W:=W_{1}\times \ldots
\times W_{n}\text{, \ }V:=V_{1}\times \ldots \times V_{n}.
\end{equation*}%
Write the element in $V$ as $x=(x_{1},\ldots ,x_{n})$ where $x_{i}\in V_{i}$%
. Set the norm $\left\Vert \cdot \right\Vert _{V}$ on $V$ as%
\begin{equation}
\left\Vert x\right\Vert _{V}:=\max_{i\in \lbrack n]}\left\{ \left\Vert
x_{i}\right\Vert _{V_{i}}\right\} ,\text{ \ }\forall x\in V.
\label{def-V-norm}
\end{equation}%
For two Banach spaces $E_{1}$ and $E_{2}$, and some $g:E_{1}\rightarrow
E_{2} $, we say $g(r)=o(r)$, if there exists a mapping $\varepsilon
:E_{1}\rightarrow E_{2}$ such that $\lim_{\left\Vert r\right\Vert
_{E_{1}}\rightarrow 0}\left\Vert \varepsilon (r)\right\Vert _{E_{2}}=0$, and 
$g(r)=\left\Vert r\right\Vert _{E_{1}}\varepsilon (r)$. We introduce the
definition of twice Fr\'{e}chet differentiability as follows.

\begin{definition}
\label{def-differentiable}\bigskip A functional $f(\cdot )$\thinspace $%
:V\rightarrow \mathbb{R}$ is twice Fr\'{e}chet differentiable on $V$, if

(1) For each $x\in V$ there exists a bounded linear functional $f^{\prime
}(x)(\cdot ):V\rightarrow \mathbb{R}$ such that%
\begin{equation}
f(x+r)-f(x)-f^{\prime }(x)(r)=o(r).  \label{def-f-differentiable}
\end{equation}%
For each $i\in \lbrack n]$, we define the partial differential $%
f_{i}(x)(\cdot ):V_{i}\rightarrow \mathbb{R}$ as%
\begin{equation*}
f_{i}(x)(r_{i}):=f^{\prime }(x)((0,\ldots ,r_{i},\ldots ,0)),
\end{equation*}%
where $(0,\ldots ,r_{i},\ldots ,0)\in V$ is an element with the $i$th
coordinate $r_{i}\in V_{i}$ and $0$ otherwise.

(2) Moreover, $\forall z_{i}\in V_{i}$, $f_{i}(\cdot
)(z_{i}):V_{i}\rightarrow \mathbb{R}$ is Fr\'{e}chet differentiable. That
is, $\forall x\in V$ there exists a bounded linear functional $f_{i}^{\prime
}(x)(z_{i},\cdot ):V\rightarrow \mathbb{R}$ such that%
\begin{equation*}
f_{i}(x+r)(z_{i})-f_{i}(x)(z_{i})-f_{i}^{\prime }(x)(z_{i},r)=o(r).
\end{equation*}%
Similarly, $\forall i,j\in \lbrack n]$ and $z_{i}\in V_{i}$, we define the
twice partial differential $f_{ij}(x)(z_{i},\cdot ):V_{j}\rightarrow \mathbb{%
R}$ as%
\begin{equation*}
f_{ij}(x)(z_{i},r_{j}):=f_{i}^{\prime }(x)(z_{i},(0,\ldots ,r_{j},\ldots
,0)).
\end{equation*}
\end{definition}

For more properties about Fr\'{e}chet differentials, see \cite{CNVS}. We
define the operator norms of the first two partial derivatives of $f(x)$ as%
\begin{eqnarray*}
\left\Vert f_{i}(x)\right\Vert &:=& \sup_{\left\Vert r_{i}\right\Vert
_{V_{i}}\leq 1}\left\vert f_{i}(x)(r_{i})\right\vert \text{, } \\
\left\Vert f_{ij}(x)\right\Vert &:=& \sup_{\max \{\left\Vert
r_{j}\right\Vert _{V_{j}},\left\Vert z_{i}\right\Vert _{V_{i}}\}\leq
1}\left\vert f_{ij}(x)(z_{i},r_{j})\right\vert ,\text{ \ }\forall i,j\in
\lbrack n].
\end{eqnarray*}%
Denote by $\left\vert f(x)\right\vert $ the absolute value of $f(x)$. We
assume that there exists $a,b_{i},c_{ij}>0$ such that $\forall x\in W$, 
\begin{equation*}
\left\vert f(x)\right\vert \leq a\text{, \ }\left\Vert f_{i}(x)\right\Vert
\leq b_{i}\text{, \ \ }\left\Vert f_{ij}(x)\right\Vert \leq c_{ij}\text{, \
\ }\forall i,j\in \lbrack n].
\end{equation*}%
Since $W_{i}$'s are assumed to be compact, we can find $M>0$ such that each $%
W_{i}$ satisfies%
\begin{equation}
\forall z_{i}^{(1)},z_{i}^{(2)}\in W_{i}\text{, \ \ }\left\Vert
z_{i}^{(1)}-z_{i}^{(2)}\right\Vert _{V_{i}}\leq M.  \label{intro-M}
\end{equation}%
Denoting by $m(\nu _{i})\in V_{i}$ the mean of $\nu _{i}$, namely the unique
point $m$ such that%
\begin{equation}
\int_{V_{i}}h\left( z\right) d\nu _{i}\left( z\right) =h(m)\text{, }\forall 
\text{ bounded linear functional }h:V_{i}\rightarrow \mathbb{R}\text{.}
\label{def-mi}
\end{equation}%
The existence of $m(\nu _{i})$ is guaranteed by the fact that $\mu _{i}$ is
supported on the compact set $W_{i}$, for example see \cite[Chapter 2]{PIB}.
Then, for any product measure $\nu =\nu _{1}\times \nu _{2}\times \ldots
\times \nu _{n}$ on $W$, let%
\begin{equation}
m(\nu ):=(m(\nu _{1}),\ldots ,m(\nu _{n}))\text{.}  \label{def-m}
\end{equation}%
Fixing some $\epsilon >0$, assume that there exists a finite set $\mathcal{D}%
(\epsilon )=\{d^{(\alpha )}=(d_{1}^{(\alpha )},\ldots ,d_{n}^{(\alpha
)}),\alpha \in I\}$ (where $I$ is the index set, and for each $\alpha \in
I,i\in \lbrack n]$, $d_{i}^{(\alpha )}$ is a bounded linear functional from $%
V_{i}$ to $\mathbb{R}$) such that for any $x\in W$, there exists a $%
d=(d_{1},\ldots ,d_{n})\in \mathcal{D}(\epsilon )$ satisfying%
\begin{equation}
\sum_{i=1}^{n}\left\Vert f_{i}(x)-d_{i}\right\Vert ^{2}\leq \epsilon ^{2}n.
\label{covering_condition}
\end{equation}%
Denote by $\left\vert \mathcal{D}(\epsilon )\right\vert $ the cardinality of 
$\mathcal{D}(\epsilon )$. Following is the main theorem, which gives upper
and lower bounds of the mean field approximation for $\log \mathbb{E}_{\mu
}[e^{f(X)}]$ where $X\thicksim \mu $.

\begin{theorem}
\label{main-theorem}Under the above setting, we have%
\begin{equation}
\log \int_{W}e^{f(x)}d\mu (x)\leq \max_{\nu \ll \mu ,\nu =\nu _{1}\times \nu
_{2}\times \ldots \times \nu _{n}}\left\{ f(m(\nu ))-\sum_{i=1}^{n}D(\nu
_{i}\parallel \mu _{i})\right\} +B_{1}+B_{2}+\log 2+\log \left\vert \mathcal{%
D}(\epsilon )\right\vert ,  \label{thm-1-upper-bound}
\end{equation}%
where%
\begin{eqnarray}
B_{1}&:= &4\left( M^{2}\left(
a\sum_{i=1}^{n}c_{ii}+\sum_{i=1}^{n}b_{i}^{2}\right)
+M^{3}\sum_{i,j=1}^{n}b_{i}c_{ij}+M^{4}\left(
a\sum_{i,j=1}^{n}c_{ij}^{2}+\sum_{i,j=1}^{n}b_{i}b_{j}c_{ij}\right) \right)
^{\frac{1}{2}},  \label{def-B} \\
B_{2}&:= &4\left( \sum_{i=1}^{n}b_{i}^{2}+\epsilon ^{2}n\right) ^{\frac{1}{2}%
}\left( M^{3}\left( \sum_{i=1}^{n}c_{ii}^{2}\right) ^{\frac{1}{2}}+M^{2}n^{%
\frac{1}{2}}\epsilon \right) +\sum_{i=1}^{n}M^{2}c_{ii}+Mn\epsilon .
\label{def-D}
\end{eqnarray}%
Moreover,%
\begin{equation}
\log \int_{W}e^{f(x)}d\mu (x)\geq \max_{\nu \ll \mu ,\nu =\nu _{1}\times \nu
_{2}\times \ldots \times \nu _{n}}\left\{ f(m(\nu ))-\sum_{i=1}^{n}D(\nu
_{i}\parallel \mu _{i})\right\} -\frac{M^{2}}{2}\sum_{i=1}^{n}c_{ii}.
\label{thm-1-lower-bound}
\end{equation}
\end{theorem}

%\begin{remark}[of Theorem \protect\ref{main-theorem}]
%A slight improvement is: Let $B_{3}:=\sup_{x\in W}\left\Vert
%f_{i}(x)\right\Vert ^{2}$, and then obviously $B_{3}\leq
%\sum_{i=1}^{n}b_{i}^{2}$. We can improve $B_{2}$ to $\widetilde{B_{2}}$
%defined as follows%
%\begin{equation*}
%\widetilde{B_{2}}:=4\left( B_{3}+\epsilon ^{2}n\right) ^{\frac{1}{2}}\left(
%M^{3}\left( \sum_{i=1}^{n}c_{ii}^{2}\right) ^{\frac{1}{2}}+M^{2}n^{\frac{1}{2%
%}}\epsilon \right) +\sum_{i=1}^{n}M^{2}c_{ii}+Mn\epsilon .
%\end{equation*}
%\end{remark}

Theorem \ref{main-theorem} is an extension of \cite[Theorem 1.5]{NLD}. If $%
\mu _{i}$'s are Bernoulli distribution with parameter $\frac{1}{2}$, Theorem %
\ref{main-theorem} is merely \cite[Theorem 1.5]{NLD} with slight
modifications. The main challenge here is to avoid the special properties of
the Bernoulli distribution and the hypercube, which are used in the proof of 
\cite[Theorem 1.5]{NLD}. For example, letting $\widetilde{\mu }$ be a
measure such that $\frac{d\widetilde{\mu }}{d\mu }(x)\propto e^{f(x)}$, in 
\cite[preceding Lemma 3.1]{NLD} the authors utilize the explicit formula for 
$\widehat{X}_{i}:=E_{\widetilde{\mu }}\left[ X_{i}\mid X_{j},j\neq i\right] $
in case of Bernoulli $\{X_{i}\}$ when bounding $f(X)-f(\widehat{X})$.
Lacking such a simple formula here requires a more sophisticated analysis of
the error induced by approximating $f(X)$ by $f(\widehat{X})$. For another
instance, in \cite{NLD}, for any point $p$ in the hypercube one has a
product Bernoulli measure $\nu ^{p}$ such that $\nu ^{p}\ll \mu $ and $%
m\left( \nu ^{p}\right) =p$. Lacking such explicit description of $\nu ^{p}$
for all $p\in W$, we instead manage to carry the proof while restricting to $%
\nu ^{p}$ for $p$ in a finite subset of $W$, for which the explicit
description (\ref{d_equation}) exists. See detailed discussions on the
difference from \cite{NLD}, important part in this extension, and the
outline of the proof of Theorem \ref{main-theorem} in Section \ref{sketch}.
The full proof of Theorem \ref{main-theorem} is given in Section \ref%
{section-proof-main-theorem}.

\subsection{\textbf{Applications}\label{section-applications}}

We provide three applications of our framework. The first two of them are
large deviations of subgraph counts in random graph, and the third one is
the mean field approximation for vector spin models.

\subsubsection{\textbf{Monochromatic subgraph counts in edge-colored
complete graphs} \label{ex-multi-color}}

The edge colored complete graph is an important object in combinatorics, for
example see Ramsey's Theorem. People have studied this kind of graphs from
different perspectives, for example see \cite{ECGRAPH-1}, \cite{ECGRAPH-2}
and \cite{ECGRAPH-3}. On the other hand, the large deviations for subgraph
counts in random graph has been studied a lot in probability, for example
see \cite{SKR04}, \cite{BGLZ15} and \cite{C16}. In this example, we consider
the large deviation for the monochromatic subgraph counts in an edge colored
random graph. More precisely, we consider a complete graph $G$ with $N$
vertices, and assume that each edge of $G$ has a color which is i.i.d.
uniformly chosen from $l$ different colors. Take any fixed finite simple
graph $H$. 
%we denote $T(H,G)$ as the number of homomorphisms of $H$ into $G$
%whose edges are of the same color. 
We investigate the large deviation of the number of homomorphisms of $H$
into $G$ whose edges are of the same color. We formulate this problem as
follows: consider a random vector $X=(X_{ij})_{1\leq i<j\leq N}$, where $%
X_{ij}$'s are i.i.d. chosen from the set $\Lambda :=\{(1,0,\ldots
,0),(0,1,\ldots ,0),\ldots ,(0,0,\ldots ,1)\}$ (where there are $l$ elements
in $\Lambda$ and the length of each element is $l$). Regard each element in $%
\Lambda $ as a color, and regard $X_{ij}$ as the color of the edge $\{i,j\}$%
. Then $X$ corresponds to a coloring on $G$. Let $m$ be the number of edges
of $H$, $\Delta $ be the maximum degree of $H$, and $k$ be the number of
vertices of $H$. For convenience we let the vertex set of $H$ be $%
\{1,\ldots,k\}$, and denote by $E$ the edge set of $H$. For $%
x=(x_{ij})_{1\leq i<j\leq N}$ where $x_{ij}\in \mathbb{R}^{l}$, define%
\begin{equation}
T(x):=\sum_{q_{1},q_{2},\ldots q_{k}\in \left[ N\right] }\sum_{s=1}^{l}\prod%
\limits_{\{r,r^{\prime }\}\in E}x_{q_{r}q_{r^{\prime }}s},  \label{def-T-x}
\end{equation}%
where $x_{q_{r}q_{r^{\prime }}s}$ is the $s$th coordinate of $%
x_{q_{r}q_{r^{\prime }}}$ (recall that $x_{q_{r}q_{r^{\prime }}}\in \Lambda $
is a vector with length $l$), $x_{ij}$ is interpreted as $x_{ji}$ if $i>j$,
and $x_{ii}$ is interpreted as the $0$ vector in $\mathbb{R}^l$ for all $i$.
It is easy to check that for coloring $X$, $T(X)$ is the number of
homomorphisms of $H$ in $G$ with same color edges. Denote by $o(1)$ a
quantity which goes to $0$ as $N$ goes to $\infty $. We show the following
large deviation result for $T(X).$

\begin{theorem}
\label{Prop-ex-multi-color}For $T(X)$ as above and any $u>1$, as $%
N\rightarrow \infty $ we have%
%\begin{eqnarray*}
%&&\mathbb{P}(T(X) \geq u\mathbb{E[}T(X)])\leq \exp \left( -\psi
%_{l}(u)(1+o(1))\right) \text{, \ \ when }l\leq N^{1/(19+8m+21\Delta )}, \\ and\\
%&&\mathbb{P}(T(X) \geq u\mathbb{E[}T(X)])\geq \exp \left( -\psi
%_{l}(u)(1+o(1))\right) \text{, \ \ when }l\leq N^{1/(2\Delta +m+2))}(\log
%N)^{-1},
%\end{eqnarray*}%
\begin{equation*}
\mathbb{P}(T(X)\geq u\mathbb{E[}T(X)])\leq \exp \left( -\psi
_{l}(u)(1+o(1))\right) \text{, \ \ when }l\leq N^{1/(19+8m+21\Delta )},
\end{equation*}%
and 
\begin{equation*}
\mathbb{P}(T(X)\geq u\mathbb{E[}T(X)])\geq \exp \left( -\psi
_{l}(u)(1+o(1))\right) \text{, \ \ when }l\leq N^{1/(2\Delta +m+2))}(\log
N)^{-1},
\end{equation*}%
where%
\begin{equation}
\psi _{l}(u):=\inf \{\sum_{1\leq i<j\leq N}\sum_{s=1}^{l}x_{ijs}\log \frac{%
x_{ijs}}{1/l}:x_{ij}\in W_{0},\text{ }T((x_{ij})_{1\leq i<j\leq N})\geq u%
\mathbb{E[}T(X)]\},  \label{def-ex1-variation}
\end{equation}%
and%
\begin{equation}
W_{0}:=\{(z_{1},\ldots ,z_{l}):\text{ }\sum_{i=1}^{l}z_{i}=1,\text{ \ }%
z_{i}\geq 0\text{ \ }\forall i\in \lbrack l]\}.  \label{def_W}
\end{equation}
\end{theorem}

Theorem \ref{Prop-ex-multi-color} provides the large deviation rate function
for $T(X)$ via the variational problem (\ref{def-ex1-variation}), in the
case that the number of colors $l$ not increasing with $N$ faster than
certain polynomial speed. We give the proof of Theorem \ref%
{Prop-ex-multi-color} in Section \ref{section-ex1-proof}.

\subsubsection{\textbf{Triangle counts with continuous edge weights}\label%
{ex-continuous}}

The large deviation principle for the triangle counts in random graph has
been studied for a long time. People study this problem for both dense Erd%
\H{o}s-R\'{e}nyi random graph $G(N,p)$, in which $p$ is fixed (\cite{CV11}),
and sparse Erd\H{o}s-R\'{e}nyi random graph $G(N,p)$, in which $p$ goes to $%
0 $ as $N$ goes to $\infty $ (\cite{KV04}, \cite{SKR04}, \cite{DK12}, \cite%
{LZ14}, \cite{NLD}, \cite{RE}). See Chatterjee \cite{C16} for more
discussions and references. Here we consider the continuous version of the
triangle counts problem in the dense random graph. That is, let $G$ be a
complete graph with $N$ vertices. Let $X=(X_{ij})_{1\leq i<j\leq N}$ where $%
X_{ij}$'s are i.i.d. from $U(0,1)$, the uniform distribution on $[0,1]$. For
each $1\leq i<j\leq N$, we assign a weight $X_{ij}$ to the edge $\{i,j\}$.
For $x=(x_{ij})_{1\leq i<j\leq N}$, we define%
\begin{equation*}
T(x):= \frac{1}{6} \sum_{i,j,k \in \left[ N\right] } x_{ij}x_{jk}x_{ki},
\end{equation*}%
where we interpret $x_{ij}=x_{ji}$ if $i>j$, and $x_{ii}=0$ for all $i\in %
\left[ N\right] $. Then $T(X)$ is the number of weighted triangles in $G$
for weights $X$. For any $a\in (0,1),$ we denote by $\nu ^{a}$ the truncated
exponential distribution on $[0,1]$ with mean $a$, that is, the distribution
whose density $p_{\nu ^{a}}(\cdot )$ is%
\begin{equation*}
p_{\nu ^{a}}(z)=\frac{\lambda _{a}e^{-\lambda _{a}z}}{1-e^{-\lambda _{a}}}%
\text{ for }z\in \left( 0,1\right) \text{, with }\lambda _{a}\text{ such
that }\int_{0}^{1}p_{\nu ^{a}}(z)dz=a.
\end{equation*}%
By direct calculation, the KL divergence between $\nu ^{a}$ and $U(0,1)$ is%
\begin{equation*}
D(\nu ^{a}||U(0,1))=\int_{0}^{1}\frac{\lambda _{a}e^{-\lambda _{a}x}}{%
1-e^{-\lambda _{a}}}\log (\frac{\lambda _{a}e^{-\lambda _{a}x}}{%
1-e^{-\lambda _{a}}})dx=-1+\frac{\lambda _{a}e^{-\lambda _{a}}}{%
1-e^{-\lambda _{a}}}+\log (\frac{\lambda _{a}}{1-e^{-\lambda _{a}}}).
\end{equation*}%
Let $n=N(N-1)/2$, the number of edges in $G$. Define%
\begin{equation*}
\psi _{n}(u):=\inf \{\sum_{1 \leq i<j \leq N}(-1+\frac{\lambda
_{y_{ij}}e^{-\lambda _{y_{ij}}}}{1-e^{-\lambda _{y_{ij}}}}+\log (\frac{%
\lambda _{y_{ij}}}{1-e^{-\lambda _{y_{ij}}}})):y_{ij}\in (0,1),\text{ }%
T((y_{ij})_{1\leq i<j\leq N})\geq u\mathbb{E[}T(X)]\}.
\end{equation*}%
We show that

\begin{theorem}
\label{Prop-ex-continuous}For $T(X)$ as above and any $1<u<8$, we have%
\begin{equation*}
\mathbb{P}(T(X)\geq u\mathbb{E[}T(X)])=\exp \left( -\psi
_{n}(u)(1+o(1))\right) \text{ as }N\rightarrow \infty \text{.}
\end{equation*}
\end{theorem}

We give the proof of Theorem \ref{Prop-ex-continuous} in Section \ref%
{section-ex2-proof}.

\begin{remark}[of Theorem \protect\ref{main-theorem}]
\label{discuss-improvement}The bounds in Theorem \ref{main-theorem} are not
guaranteed and have no reason to be optimal; they could be improved case by
case by utilizing particular structures of specific problems. 
%For instance,
%in the proof of Theorem \ref{main-theorem}, $B$ appears where we want to
%control $f(x)-f(\widehat{x})$ and $\sum_{i=1}^{n}f_{i}(x^{(i)})(x_{i}-%
%\widehat{x}_{i})$, and it is possible to find better bounds for these terms
%in specific problems. Similar arguments apply to $D$. Further, it is
%possible to find suitable replacements for other elements in the proof of
%Theorem \ref{main-theorem}. 
We provide the following example to show this.
\end{remark}

\subsubsection{\textbf{Mean field approximation on a class of vector spin
models\label{ex-3}}}

Mean field approximation is an important method derived from Physics, and it
has been applied to many different fields. See \cite{AMFM00} or \cite{AS}
for an introduction to this method. Like other methods in statistical
physics, its mathematical rigor is not guaranteed and needs to be verified
for specific models. In \cite{AS} the universality of the mean field
approximation for a class of Potts model is verified. Our next theorem
extends the result in \cite{AS} to a more general setting. We introduce some
notations first. Let $X_{i}$'s be i.i.d. random variables with corresponding
distributions $\mu _{i}$'s supported on a compact set $W_{1}$ in $\mathbb{R}%
^{N}$ for some $N\geq 1$. Define the product measure as $\mu :=\mu
_{1}\times \ldots \times \mu _{n}$. Let $J$ be a real symmetric $N\times N$
matrix, $h$ be a real vector with length $N$, and for each $n\in \mathbb{Z}%
^{+}$ let $A_{n}$ be a real symmetric $n\times n$ matrix. Define the
Hamiltonian $H_{n}^{J,h}(\cdot ):(\mathbb{R^{N}})^{n}\rightarrow \mathbb{R}$
such that for any $x=(x_{1},\ldots ,x_{n})\in (\mathbb{R^{N}})^{n}$ 
\begin{equation}
H_{n}^{J,h}(x):=\frac{1}{2}\sum_{i,j=1}^{n}A_{n}(i,j)x_{i}^{T}Jx_{j}+%
\sum_{i=1}^{n}x_{i}^{T}h\text{.}  \label{def-general-H}
\end{equation}%
For a sequence $\{c_{n}\}_{n\geq 1}$ and a positive sequence $\{a_n\}$, we
say $c_{n}=o(a_n)$ if $\lim_{n\rightarrow \infty }c_{n}/{a_n}=0$, and $%
c_{n}=O(a_n)$ if $\limsup_{n\rightarrow \infty }\left\vert c_{n}\right\vert /%
{a_n}<\infty $. We have the following theorem.

\begin{theorem}
\label{theorem-as-new}If the sequence of matrices $A_{n}$ satisfies%
\begin{equation}
tr(A_{n}^{2})=o(n)\,\,\text{ and }\,\,\sup_{x\in \lbrack 0,1]^{n}}\sum_{i\in
\lbrack n]}\left\vert \sum_{j\in \lbrack n]}A_{n}(i,j)x_{j}\right\vert =O(n),
\label{ex-new-condition}
\end{equation}%
then{\small 
\begin{equation}
\lim_{n\rightarrow \infty }\frac{1}{n}\left[ \log
\int_{W_{1}^{n}}e^{H_{n}^{J,h}(x)}d\mu (x)-\max_{\nu \ll \mu ,\nu =\nu
_{1}\times \nu _{2}\times \ldots \times \nu _{n},}\left\{ H_{n}^{J,h}(m(\nu
))-\sum_{i=1}^{n}D(\nu _{i}\parallel \mu _{i})\right\} \right] =0.
\label{eq-theorem-as-new}
\end{equation}%
}
\end{theorem}

\begin{remark}[of Theorem \protect\ref{theorem-as-new}]
If we let $\mu _{i}$'s be the uniform distribution on $\{(1,0,\ldots
,0),\ldots ,(0,0,\ldots ,1)\}$ (each element belongs to $\mathbb{R}^N$ for $%
N \geq 2$ and has a unique nonzero entry), then we get the Potts model, and
Theorem \ref{theorem-as-new} is merely Theorem 1.1 in \cite{AS}.
\end{remark}

Theorem \ref{theorem-as-new} covers a large class of models in statistical
physics. In the simple case of $A_{n}(i,j)=1/n$, it is easy to verify that
condition (\ref{ex-new-condition}) holds, and $N=1,2,3$ correspond to the
mean field Curie-Weiss model, XY model and Heisenberg model respectively.
The validity of the mean field approximation for these mean field models has
long been known, for example see \cite{KT} and \cite{DZ10}. The more
difficult case is when $A_{n}(i,j)$ are not same, see examples and
discussions in \cite[Section 1.3]{AS}. A direct application of Theorem \ref%
{theorem-as-new} is letting $\mu _{i}$ be the uniform distribution on the
unit sphere $S^{N-1}$, which is often studied in statistical physics and is
not covered by \cite{AS}.

If we directly apply Theorem \ref{main-theorem} to the setting above, we
will find that (\ref{eq-theorem-as-new}) is stronger than what we can get.
In order to prove Theorem \ref{theorem-as-new}, we need to incorporate the
special properties of $H_{n}^{J,h}$. We give the proof of Theorem \ref%
{theorem-as-new}\ in Section \ref{section-new-ex-proof}.

\bigskip

We give the proof outline of Theorem \ref{main-theorem} in Section \ref%
{sketch} below, including detailed discussions on the differences from \cite%
{NLD} and important parts in our extensions. The full proof of Theorem \ref%
{main-theorem} is provided in Section \ref{section-proof-main-theorem}. The
proofs of three applications are given in Section \ref%
{section-proof-applications}.

\section{Proof Outline of Theorem \protect\ref{main-theorem}\label{sketch}}

We proceed to sketch the key part of Theorem \ref{main-theorem}, namely
proving the upper bound (\ref{thm-1-upper-bound}), together with the
differences from the proof in \cite{NLD} (see Section \ref%
{section-proof-main-theorem-lower-bound} for the much easier proof of the
lower bound (\ref{thm-1-lower-bound})).

(1) We define a measure $\widetilde{\mu }$ supported on $W$ such that%
\begin{equation}
\frac{d\widetilde{\mu }}{d\mu }(x):=\frac{e^{f(x)}}{\int_{W}e^{f(x)}d\mu (x)}%
\text{, }\forall x\in W.  \label{def-miu-hat}
\end{equation}%
We define $\widehat{\mathrm{x}}_{i}(\cdot ):V\rightarrow W_{i}$ and $%
\widehat{\mathrm{x}}(\cdot ):V\rightarrow W$, such that for every $%
x=(x_{1},\ldots ,x_{n})\in V$,%
\begin{equation}
\widehat{\mathrm{x}}_{i}(x):=\frac{\int_{W_{i}}z_{i}e^{f(x_{1},\ldots
,x_{i-1},z_{i},x_{i+1},\ldots ,x_{n})}d\mu _{i}(z_{i})}{%
\int_{W_{i}}e^{f(x_{1},\ldots ,x_{i-1},z_{i},x_{i+1},\ldots ,x_{n})}d\mu
_{i}(z_{i})}\text{ \ \ and \ \ }\widehat{\mathrm{x}}(x):=(\widehat{\mathrm{x}%
}_{1}(x),\ldots ,\widehat{\mathrm{x}}_{n}(x)).  \label{def-x-hat}
\end{equation}%
For simplicity, we write $\widehat{x}$ and $\widehat{x}_{i}$ for $\widehat{%
\mathrm{x}}(x)$ and $\widehat{\mathrm{x}}_{i}(x)$. For $x\in W$, $\widehat{x}%
_{i}$ is merely $\mathbb{E}_{\widetilde{\mu }}\left[ X_{i}\mid X_{j}=x_{j}%
\text{ for }j\neq i\right] $. The existence of $\widehat{x}$ is guaranteed
by the fact that $W_{i}$ is compact, and obviously $\widehat{x}\in W$ since $%
W_{i}$ is convex. We first do the approximation%
\begin{equation}
f(x)\thickapprox f(\widehat{x}).  \label{eq-approx-1}
\end{equation}
%In the sketch
% we will not bother to make precise the meaning of "$\thickapprox$", 
In this sketch we write $L \thickapprox R$ if under $\widetilde{\mu }$ with
high probability $|L-R|$ is controlled, we will not bother to make rigorous
the meaning of $\thickapprox$. \newline
In \cite{NLD}, since each $\mu _{i}$ is supported on $\{0,1\}$, $\widehat{x}$
has the good expression \cite[the expression above Lemma 3.1]{NLD}:%
\begin{equation}
\widehat{x}_{i}=\frac{1}{1+e^{-\Delta _{i}f(x)}},  \label{eq-NLD-xhat}
\end{equation}%
where $\Delta _{i}f(x)$ is the discrete derivative defined as follows 
\begin{equation*}
\Delta _{i}f(x):=f(x_{1},\ldots ,x_{i-1},1,x_{i+1},\ldots
,x_{n})-f(x_{1},\ldots ,x_{i-1},0,x_{i+1},\ldots ,x_{n}).
\end{equation*}%
In our case we do not have a good expression as (\ref{eq-NLD-xhat}).

(2) The next step is to construct a covering set $\mathcal{D}^{\prime
}(\epsilon )$ of $\left\{ \widehat{x}:\text{ }x\in W\right\} $, such that
for each $x\in W$, there exists some $p^{x}=(p_{1}^{x},\ldots ,p_{n}^{x})\in 
\mathcal{D}^{\prime }(\epsilon )$ which is close to $\widehat{x}$. 
%in some sense (the choice of $p^{x}$ might be not unique, in which case
%we just choose any one of the possible choices). 
Consequently we have%
\begin{equation}
f(\widehat{x})\thickapprox f(p^{x}).  \label{eq-purpose-Dprime}
\end{equation}

In \cite{NLD}, the covering set $\mathcal{D}^{\prime }(\epsilon )$ is
constructed by applying a function $u(x)=1/(1+e^{-x})$ on each point in $%
\mathcal{D}(\epsilon )$ (\cite[3 lines below (3.16)]{NLD}). This makes sense
because $\mathcal{D}(\epsilon )$ is the covering set of the gradient of $%
f(x) $, and $\widehat{x}_{i}$ has the expression (\ref{eq-NLD-xhat}). 
%(\ref{eq-NLD-xhat}) helps to construct the covering set $%
%\mathcal{D}^{\prime }(\epsilon )$, by applying a function $u(x)=1/(1+e^{-x})$
%on each point in $\mathcal{D}(\epsilon )$ (\cite[3 lines below (3.16)]{NLD})
%as follows%
%\begin{equation}
%\mathcal{D}^{\prime }(\epsilon ):=\{p:p=(p_{i})_{i\in \lbrack
%n]},p_{i}=u(d_{i}(1)),d=(d_{i})_{i\in \lbrack n]}\in \mathcal{D}(\epsilon
%)\},  \label{def-D-prime}
%\end{equation}%
%
%
%
%Note that $d_{i}$ is defined as functional in our notation. 
%(in \cite{NLD} the expression of $p_{i}$ is $p_{i}=u(d_{i})$. Here we
%replace $d_{i}$ by $d_{i}(1)$, since $d_{i}$ defined in our paper is a
%functional on $\{0,1\}$, rather than a real number). 
%In above definition we can apply $u(\cdot )$ to $d_{i}$, since in \cite{NLD} $x_{i}$
%is supported on $\{0,1\}$, and therefore $d_{i}$ defined in our paper (which is a functional)
%corresponds to its value on the point $\{1\}$, which is a real number. 
Special properties of this explicit construction is used in \cite{NLD}, such
as $\left\vert u^{\prime }(x)\right\vert \leq 1/4$. In our case we construct 
$\mathcal{D}^{\prime }(\epsilon )$ in the general setting.

(3) Next, for each $i$ and $p=(p_{1},\ldots ,p_{n})\in \mathcal{D}^{\prime
}(\epsilon )$ we construct a measure $\nu _{i}^{p}$ supported on $W_{i}$,
such that $\nu _{i}^{p}\ll \mu _{i}$, $m(\nu _{i}^{p})=p_{i}$, and the
following approximation holds%
\begin{equation}
-\sum_{i=1}^{n}D(\nu _{i}^{p^{x}}\parallel \mu _{i})+\sum_{i=1}^{n}\log (%
\frac{d\nu _{i}^{p^{x}}}{d\mu _{i}}(x)) \thickapprox 0.
\label{eq-purpose-vy}
\end{equation}

In \cite{NLD}, $\mu _{i}$ is Bernoulli$(\frac{1}{2})$ (the Bernoulli
distribution with parameter $\frac{1}{2}$). Therefore for any $%
y=(y_{1},\ldots ,y_{n})\in \lbrack 0,1]^{n}$, the unique measure $\nu
_{i}^{y}$ with $\nu _{i}^{y}\ll \mu _{i}$ and $m(\nu _{i}^{y})=y_{i}$ is
just Bernoulli$(y_{i})$. Hence one can write down the explicit form of the
KL divergence between $\nu _{i}^{y}$ and $\mu _{i}$ as%
\begin{equation*}
D(\nu _{i}^{y}\parallel \mu _{i})=y_{i}\log y_{i}+(1-y_{i})\log
(1-y_{i})+\log 2.
\end{equation*}%
In this way, $-\sum_{i=1}^{n}D(\nu _{i}^{p^{x}}\parallel \mu
_{i})+\sum_{i=1}^{n}\frac{d\nu _{i}^{p^{x}}}{d\mu _{i}}(x)$ becomes \cite[%
(3.13)]{NLD}, which has a good form to analyze. In our case, we build the
measure $\nu _{i}^{p^{x}}$ in Section \ref{section-pre-analysis}, and we
show several general properties of this kind of measures, which help us to
prove our approximation.

(4) Combining (\ref{eq-approx-1}), (\ref{eq-purpose-Dprime}) and (\ref%
{eq-purpose-vy}), we get the following approximation%
\begin{equation}
f(x)\thickapprox f(p^{x})-\sum_{i=1}^{n}D(\nu _{i}^{p^{x}}\parallel \mu
_{i})+\sum_{i=1}^{n}\log (\frac{d\nu _{i}^{p^{x}}}{d\mu _{i}}(x)).
\label{eq-approximation}
\end{equation}

In \cite{NLD}, to bound the error of the above approximation, the authors
decompose the error into $f(x)-f(\widehat{x})$ and \cite[(3.13)]{NLD}, which
does not work in the general case here. In our proof, we find the
decomposition ((\ref{def-I}) and (\ref{def-II})) that works in general.

(5) Note that if we fix $y\in W$, then by the fact that $\int_{W_{i}}\frac{%
d\nu _{i}^{p^{y}}}{d\mu _{i}}(x)d\mu _{i}(x)=1$ we get%
\begin{equation*}
\int_{W}e^{f(p^{y})-\sum_{i=1}^{n}D(\nu _{i}^{p^{y}}\parallel \mu
_{i})+\sum_{i=1}^{n}\log \frac{d\nu _{i}^{p^{y}}}{d\mu _{i}}(x)}d\mu
(x)=f(p^{y})-\sum_{i=1}^{n}D(\nu _{i}^{p^{y}}\parallel \mu _{i}).
\end{equation*}%
Therefore, with above approximations we have that{\small 
\begin{eqnarray*}
\log \int_{W}e^{f(x)}d\mu (x) &=&\log
\int_{W}e^{f(p^{x})-\sum_{i=1}^{n}D(\nu _{i}^{p^{x}}\parallel \mu
_{i})+\sum_{i=1}^{n}\log \frac{d\nu _{i}^{p^{x}}}{d\mu _{i}}(x)}d\mu (x)+%
\text{error terms} \\
&\leq &\log \sum_{p\in \mathcal{D}^{\prime }(\epsilon )}\left(
f(p)-\sum_{i=1}^{n}D(\nu _{i}^{p}\parallel \mu _{i})\right) +\text{error
terms} \\
&\leq &\max_{\nu \ll \mu ,\nu =\nu _{1}\times \nu _{2}\times \ldots \times
\nu _{n}}\left\{ f(m(\nu ))-\sum_{i=1}^{n}D(\nu _{i}\parallel \mu
_{i})\right\} +\log \left\vert \mathcal{D}^{\prime }(\epsilon )\right\vert +%
\text{error terms,}
\end{eqnarray*}%
}where in the last inequality we use the fact that $m(\nu _{i}^{p})=p_{i}$.
The above inequality leads to the desired upper bound.

\section{Proof of Theorem \protect\ref{main-theorem}\label%
{section-proof-main-theorem}}

\subsection{\textbf{The lower bound part of Theorem \protect\ref%
{main-theorem}}}

\label{section-proof-main-theorem-lower-bound}

The idea to prove the lower bound is first to use the Gibbs variational
principle ((\ref{lower_3}) below) on any product measure $\nu $, and then to
approximate the first term on the right-hand side of (\ref{lower_3}) by $%
f(m(\nu ))$ ($m(\nu )$ is defined at (\ref{def-m})), where the error is
controlled by the norms of the second derivatives of $f$.

\begin{proof}
For any $\nu =\nu _{1}\times \nu _{2}\times \ldots \times \nu _{n}$, by the
Gibbs variational principle, we have%
\begin{equation}
\log \int_{W}e^{f(x)}d\mu (x)\geq \int_{W}f(x)d\nu (x)-D(\nu \parallel \mu ).
\label{lower_3}
\end{equation}%
Because $\nu $ and $\mu $ are both product measures, we have the following
decomposition%
\begin{equation}
D(\nu \parallel \mu )=\sum_{i=1}^{n}D(\nu _{i}\parallel \mu _{i})\text{.}
\label{lower_4}
\end{equation}%
Next we approximate $\int_{W}f(x)d\nu (x)$ by $f(m(\nu ))$. For $x\in V$, $%
i\in \lbrack n]$ and $z_{i}\in V_{i}$, define%
\begin{equation}
x_{z_{i}}^{(i)}:=(x_{1},\ldots ,x_{i-1},z_{i},x_{i+1},\ldots ,x_{n}).
\label{def-x-i-bracket}
\end{equation}%
Fix $\theta =(\theta _{1},\ldots ,\theta _{n})\in W$. For $t\in \lbrack 0,1]$%
, by the definition of $m(\nu _{i})$ (\ref{def-mi}) and the fact that $%
f_{i}(tx_{\theta _{i}}^{(i)}+(1-t)m(\nu ))(\cdot )$ is linear, we have $%
\int_{W}f_{i}(tx_{\theta _{i}}^{(i)}+(1-t)m(\nu ))(x_{i}-m(\nu _{i}))d\nu
(x)=0$, which implies that{\small 
\begin{eqnarray}
&&\left\vert \int_{W}f_{i}(tx+(1-t)m(\nu ))(x_{i}-m(\nu _{i}))d\nu
(x)\right\vert  \notag \\
&=&\left\vert \int_{W}\left( f_{i}(tx+(1-t)m(\nu ))-f_{i}(tx_{\theta
_{i}}^{(i)}+(1-t)m(\nu ))\right) (x_{i}-m(\nu _{i}))d\nu (x)\right\vert 
\notag \\
&\leq &\int_{W}c_{ii}\times \left\Vert tx_{i}-t\theta _{i}\right\Vert
_{V_{i}}\times \left\Vert x_{i}-m(\nu _{i})\right\Vert _{V_{i}}d\nu (x)\leq
tc_{ii}M^{2}.  \label{lower_2}
\end{eqnarray}%
}By (\ref{lower_2}) and the expression $f(x)-f(m(\nu
))=\sum_{i=1}^{n}\int_{0}^{1}f_{i}(tx+(1-t)m(\nu ))(x_{i}-m(\nu _{i}))dt$,
we further get%
\begin{equation}
\int_{W}\left( f(x)-f(m(\nu ))\right) d\nu (x)\geq
-\sum_{i=1}^{n}\int_{0}^{1}tc_{ii}M^{2}dt=-\frac{M^{2}}{2}%
\sum_{i=1}^{n}c_{ii}.  \label{lower_5}
\end{equation}%
Plugging (\ref{lower_4}) and (\ref{lower_5}) into (\ref{lower_3}), we get%
\begin{equation*}
\log \int_{W}e^{f(x)}d\mu (x)\geq f(m(\nu ))-\sum_{i=1}^{n}D(\nu
_{i}\parallel \mu _{i})-\frac{M^{2}}{2}\sum_{i=1}^{n}c_{ii}.
\end{equation*}%
Taking the sup over $\left\{ \nu :\nu =\nu _{1}\times \nu _{2}\times \ldots
\times \nu ,\text{ }\nu \ll \mu \right\} $ completes the proof.
\end{proof}

\subsection{\textbf{The upper bound part of Theorem \protect\ref%
{main-theorem}}}

In this subsection we prove the upper bound of Theorem \ref{main-theorem}.
In Section \ref{Section-construction-D-prime-covering}, we construct the
covering of $\left\{ \widehat{x}\text{: }x\in W\right\} $, which plays an
important role in our approximation. In Section \ref{section-pre-analysis}
we show several properties of the measure $\nu ^{p^{x}}$, which is described
in (\ref{eq-purpose-vy}) and is defined at (\ref{d_equation}). We provide
the error bound for the approximation (\ref{eq-approximation}) in Section %
\ref{section-proof-approximation}, and we summarize and finish the proof in
Section \ref{Section-upper-bound-final}.

\subsubsection{\textbf{The construction of} $\mathcal{D}^{\prime }(\protect%
\epsilon )\label{Section-construction-D-prime-covering}$}

In order to construct the covering of $\left\{ \widehat{x}\text{: }x\in
W\right\} $ (defined at (\ref{def-x-hat})), for any $d=(d_{1},d_{2},\ldots
,d_{n})\in \mathcal{D}(\epsilon )$ we construct a corresponding $%
p(d)=(p(d)_{1},p(d)_{2},\ldots ,p(d)_{n})\in W$ in the following way:
recalling that $d_{i}(\cdot )$ is a bounded linear functional from $V_{i}$
to $\mathbb{R}$, let%
\begin{equation}
p(d)_{i}:=\frac{\int_{W_{i}}z_{i}e^{d_{i}(z_{i})}d\mu _{i}(z_{i})}{%
\int_{W_{i}}e^{d_{i}(z_{i})}d\mu _{i}(z_{i})}\text{ \ \ and \ \ }%
p(d):=(p(d)_{1},\ldots ,p(d)_{n}).  \label{def-p-d}
\end{equation}%
The existence of $p(d)$ is guaranteed by the fact that $W_{i}$ is compact,
and obviously $p(d)\in W$ since $W$ is convex. Define%
\begin{equation*}
\mathcal{D}^{\prime }(\epsilon ):=\left\{ p(d):\text{ }d\in \mathcal{D}%
(\epsilon )\right\} .
\end{equation*}%
For each $x$, we choose a $d^{x}$ such that%
\begin{equation}
d^{x}\in \left\{ d\in \mathcal{D}(\epsilon )\text{ s.t. }\sum_{i=1}^{n}\left%
\Vert f_{i}(x)-d_{i}\right\Vert ^{2}\leq \epsilon ^{2}n\right\} ,
\label{def_d}
\end{equation}%
where if the set on the right-hand side contains more than one element, we
just choose any one in it and fix the choice. Using (\ref{def-p-d}) we can
further define%
\begin{equation}
p^{x}:=(p_{1}^{x},p_{2}^{x},\ldots ,p_{n}^{x})\text{, \ \ \ where \ }%
p_{i}^{x}:=p(d^{x})_{i}\text{ \ \ }\forall i\in \lbrack n]\text{.}
\label{def-p-x}
\end{equation}%
In the following we show that $\mathcal{D}^{\prime }(\epsilon )$ is a good
covering of $\left\{ \widehat{x}\text{: }x\in W\right\} $, by bounding the
term $\sum_{i=1}^{n}\left\Vert \widehat{x}_{i}-p_{i}^{x}\right\Vert
_{V_{i}}^{2}$. Recall that $d_{i}^{x}(\cdot )$ is a linear functional from $%
W_{i}$ to $\mathbb{R}$. Let%
\begin{equation*}
p_{i}^{x}(t):=\frac{%
\int_{W_{i}}z_{i}e^{tf(x_{z_{i}}^{(i)})+(1-t)d_{i}^{x}(z_{i})}d\mu
_{i}(z_{i})}{\int_{W_{i}}e^{tf(x_{z_{i}}^{(i)})+(1-t)d_{i}^{x}(z_{i})}d\mu
_{i}(z_{i})}\text{.}
\end{equation*}%
Then $p_{i}^{x}(t)$ is an interpolation between $p_{i}^{x}$ and $\widehat{x}%
_{i}$, since it is easy to verify that%
\begin{equation}
p_{i}^{x}(0)=p_{i}^{x}\text{, \ \ \ }p_{i}^{x}(1)=\widehat{x}_{i}.
\label{eq-p-i-t-0-1}
\end{equation}%
Let%
\begin{equation}
e(x,i):=\frac{\widehat{x}_{i}-p_{i}^{x}}{\left\Vert \widehat{x}%
_{i}-p_{i}^{x}\right\Vert _{V_{i}}}\text{, \ \ \ }V_{x,i}:=\{ke(x,i):k\in 
\mathbb{R}\}.  \label{def-e-1-x-i}
\end{equation}%
Then clearly $V_{x,i}$ is a $1$-dimension subspace of $V_{i}$. Define a
linear functional $g_{0}:V_{x,i}\rightarrow \mathbb{R}$ as%
\begin{equation}
g_{0}(ke(x,i))=k,  \label{def-g0}
\end{equation}%
and then obviously $\left\Vert g_{0}\right\Vert =1$. By the Hahn-Banach
theorem, we can extend $g_{0}$ to $g$, a linear functional from $V_{i}$ to $%
\mathbb{R}$ such that%
\begin{equation}
g(z_{i})=g_{0}(z_{i})\text{ \ }\forall \text{ }z_{i}\in V_{x,i},\text{ \ \ \
\ }\left\Vert g\right\Vert =\left\Vert g_{0}\right\Vert =1.  \label{def-g}
\end{equation}%
Thus for any $z_{i}^{(1)},z_{i}^{(2)}\in W_{i}$ we have%
\begin{equation}
\left\vert g(z_{i}^{(1)})-g(z_{i}^{(2)})\right\vert \leq \left\Vert
z_{i}^{(1)}-z_{i}^{(2)}\right\Vert _{V_{i}}.  \label{condition-g-z1-z2}
\end{equation}%
Using the fact that $f(\cdot )$ is bounded and Fr\'{e}chet differentiable,
and $W_{i}$ is compact, it is easy to see that $g(p_{i}^{x}(t))$ is
differentiable with respect to $t$. By the definition of $p_{i}^{x}(t)$,
after some algebra we arrive at%
\begin{equation}
\frac{dg(p_{i}^{x}(t))}{dt}=\left( \frac{%
\int_{W_{i}}g(z_{i})e^{tf(x_{z_{i}}^{(i)})+(1-t)d_{i}^{x}(z_{i})}d\mu
_{i}(z_{i})}{\int_{W_{i}}e^{tf(x_{z_{i}}^{(i)})+(1-t)d_{i}^{x}(z_{i})}d\mu
_{i}(z_{i})}\right) _{t}^{\prime }=\mathbb{E}_{\phi
_{t}^{i}}[(f(x_{Z_{i}}^{(i)})-d_{i}^{x}(Z_{i}))(g(Z_{i})-\mathbb{E}_{\phi
_{t}^{i}}[g(Z_{i})])],  \label{derivative-p-x}
\end{equation}%
where the expectation is taken with respect to $Z_{i}$, which obeys the
measure $\phi _{t}^{i}\ll \mu _{i}$ defined as%
\begin{equation*}
\frac{d\phi _{t}^{i}}{d\mu _{i}}(z_{i}):=\frac{%
e^{tf(x_{z_{i}}^{(i)})+(1-t)d_{i}^{x}(z_{i})}}{%
\int_{W_{i}}e^{tf(x_{z_{i}}^{(i)})+(1-t)d_{i}^{x}(z_{i})}d\mu _{i}(z_{i})}.
\end{equation*}%
Recall that $\theta =(\theta _{1},\ldots ,\theta _{n})$ is a fixed point in $%
W$. It is easy to check that 
\begin{equation*}
(f(tx_{z_{i}}^{(i)}+(1-t)x_{\theta _{i}}^{(i)})-d_{i}^{x}(tz_{i}+(1-t)\theta
_{i}))_{t}^{\prime }=\left( f_{i}(tx_{z_{i}}^{(i)}+(1-t)x_{\theta
_{i}}^{(i)})-d_{i}^{x}\right) (z_{i}-\theta _{i}).
\end{equation*}%
Therefore, writing the following difference as the integral of derivative,
we can see that for any $z_{i}\in W_{i}$,{\small 
\begin{eqnarray}
&&\left\vert f(x_{z_{i}}^{(i)})-d_{i}^{x}(z_{i})-(f(x_{\theta
_{i}}^{(i)})-d_{i}^{x}(\theta _{i}))\right\vert  \notag \\
&=&\left\vert \int_{0}^{1}\left( f_{i}(tx_{z_{i}}^{(i)}+(1-t)x_{\theta
_{i}}^{(i)})-d_{i}^{x}\right) (z_{i}-\theta _{i})dt\right\vert  \notag \\
&\leq &\int_{0}^{1}\left\vert \left( f_{i}(tx_{z_{i}}^{(i)}+(1-t)x_{\theta
_{i}}^{(i)})-f_{i}(x)\right) (z_{i}-\theta _{i})\right\vert
dt+\int_{0}^{1}\left\vert \left( f_{i}(x)-d_{i}^{x}\right) (z_{i}-\theta
_{i})\right\vert dt  \notag \\
&\leq &c_{ii}M^{2}+\left\Vert f_{i}(x)-d_{i}^{x}\right\Vert M.
\label{eq-bound-f-difference-4-terms}
\end{eqnarray}%
}Noting that $\mathbb{E}_{\phi _{t}^{i}}[g(Z_{i})-\mathbb{E}_{\phi _{t}^{i}}%
\left[ g(Z_{i})\right] ]=0$, we have%
\begin{equation}
\mathbb{E}_{\phi _{t}^{i}}[(f(x_{\theta _{i}}^{(i)})-d_{i}^{x}(\theta
_{i}))(g(Z_{i})-\mathbb{E}_{\phi _{t}^{i}}[g(Z_{i})])]=0.
\label{eq-construction-1}
\end{equation}%
From (\ref{intro-M}) and (\ref{condition-g-z1-z2}) it is clear that\ for
each $z_{i}^{(1)},z_{i}^{(2)}\in W_{i}$ we have \TEXTsymbol{\vert}$%
g(z_{i}^{(1)})-g(z_{i}^{(2)})|\leq M$, which implies that%
\begin{equation}
\mathbb{E}_{\phi _{t}^{i}}[|g(Z_{i})-\mathbb{E}_{\phi _{t}^{i}}\left[
g(Z_{i})\right] |]\leq M.  \label{eq-construction-2}
\end{equation}%
Subtracting $\mathbb{E}_{\phi _{t}^{i}}[(f(x_{\theta
_{i}}^{(i)})-d_{i}^{x}(\theta _{i}))(g(Z_{i})-\mathbb{E}_{\phi
_{t}^{i}}[g(Z_{i})])]$ from the right-hand side of (\ref{derivative-p-x}),
with (\ref{eq-bound-f-difference-4-terms}), (\ref{eq-construction-1}) and (%
\ref{eq-construction-2}) we have%
\begin{equation}
\left\vert \frac{dg(p_{i}^{x}(t))}{dt}\right\vert \leq
c_{ii}M^{3}+\left\Vert f_{i}(x)-d_{i}^{x}\right\Vert M^{2},
\label{lowcomplexity_3}
\end{equation}%
and consequently by (\ref{def-e-1-x-i}), (\ref{def-g0}) and (\ref{def-g}) we
see that%
\begin{equation}
\left\Vert \widehat{x}_{i}-p_{i}^{x}\right\Vert _{V_{i}}=g(\widehat{x}%
_{i}-p_{i}^{x})=g(p_{i}^{x}(1)-p_{i}^{x}(0))\leq c_{ii}M^{3}+\left\Vert
f_{i}(x)-d_{i}^{x}\right\Vert M^{2}.  \label{eq-construction-3}
\end{equation}%
Therefore from (\ref{def_d}), (\ref{eq-construction-3}) and the basic
inequalities $(a+b)^{2}\leq 2a^{2}+2b^{2}$, $(a^{2}+b^{2})^{1/2}\leq a+b$,
we have%
\begin{equation}
\left( \sum_{i=1}^{n}\left\Vert \widehat{x}_{i}-p_{i}^{x}\right\Vert
_{V_{i}}^{2}\right) ^{\frac{1}{2}}\leq \left( \sum_{i=1}^{n}\left(
c_{ii}M^{3}+\left\Vert f_{i}(x)-d_{i}^{x}\right\Vert M^{2}\right)
^{2}\right) ^{\frac{1}{2}}\leq \sqrt{2}M^{3}\left(
\sum_{i=1}^{n}c_{ii}^{2}\right) ^{\frac{1}{2}}+\sqrt{2}M^{2}n^{\frac{1}{2}%
}\epsilon .  \label{lowcomplexity_5}
\end{equation}

\subsubsection{\textbf{The} \textbf{construction and properties of the
measure} $\protect\nu ^{p}\label{section-pre-analysis}$}

Before constructing the measure $\nu ^{p}$, let us take a look at the term%
\begin{equation}
\max_{\nu \ll \mu ,\nu =\nu _{1}\times \nu _{2}\times \ldots \times \nu
_{n}}\left\{ f(m(\nu ))-\sum_{i=1}^{n}D(\nu _{i}\parallel \mu _{i})\right\} .
\label{term_max}
\end{equation}%
In order to achieve the maximum, a natural question one might ask is: when $%
(m(\nu ))$ is fixed, what is the minimum value of $\sum_{i=1}^{n}D(\nu
_{i}\parallel \mu _{i})$? For every $y=(y_{1},\ldots ,y_{n})\in W$, we
consider the following problem:%
\begin{equation}
\min \left\{ \sum_{i=1}^{n}D(\nu _{i}\parallel \mu _{i})\text{: }\nu \text{
is a product probability measure with }\nu \ll \mu \text{ and }m(\nu
)=y\right\} .\text{ }  \label{measure}
\end{equation}%
In this subsection, we show several properties of the minimizer of (\ref%
{measure}). We prove that

\begin{proposition}
\label{measure_v}If a measure $\nu ^{y}=\nu _{1}^{y}\times \nu
_{2}^{y}\times \ldots \times \nu _{n}^{y}$ satisfies that for each $i\in
\lbrack n]$,%
\begin{equation}
\nu _{i}^{y}\ll \mu _{i}\text{,\ \ }m(\nu _{i}^{y})=y_{i},\text{ \ and }%
\frac{d\nu _{i}^{y}}{d\mu _{i}}(z_{i})=e^{R_{i}(z_{i})}\text{ for a linear
functional }R_{i}(\cdot ):V_{i}\rightarrow \mathbb{R},  \label{v-exp-form}
\end{equation}%
then $\nu ^{y}$ achieves the minimum in (\ref{measure}).
\end{proposition}

\begin{proof}
For each $i$, assume that $\nu _{i}^{y}$ satisfies (\ref{v-exp-form}). For
any other measure $\widetilde{\nu }_{i}^{y}$ with $m(\widetilde{\nu }%
_{i}^{y})=y_{i}$ and $\widetilde{\nu }_{i}^{y}\ll \mu _{i}$, since $\log 
\frac{d\nu _{i}^{y}}{d\mu _{i}}(\cdot )$ is linear by (\ref{v-exp-form}), we
have%
\begin{equation}
\int_{W_{i}}\log \frac{d\nu _{i}^{y}}{d\mu _{i}}(z_{i})d\widetilde{\nu }%
_{i}^{y}(z_{i})=\int_{W_{i}}\log \frac{d\nu _{i}^{y}}{d\mu _{i}}(z_{i})d\nu
_{i}^{y}(z_{i})=D(\nu _{i}^{y}\parallel \mu _{i}).
\label{eq-newproof-preanalysis-1}
\end{equation}%
Combining (\ref{eq-newproof-preanalysis-1}) and the fact that $D(\widetilde{%
\nu }_{i}^{y}\parallel \nu _{i}^{y})\geq 0$, we have%
\begin{equation*}
0\leq D(\widetilde{\nu }_{i}^{y}\parallel \nu _{i}^{y})=\int_{W_{i}}\frac{d%
\widetilde{\nu }_{i}^{y}}{d\mu _{i}}(z_{i})\log \frac{\frac{d\widetilde{\nu }%
_{i}^{y}}{d\mu _{i}}(z_{i})}{\frac{d\nu _{i}^{y}}{d\mu _{i}}(z_{i})}d\mu
_{i}(z_{i})=D(\widetilde{\nu }_{i}^{y}\parallel \mu _{i})-D(\nu
_{i}^{y}\parallel \mu _{i}),
\end{equation*}%
and it completes the proof.
\end{proof}

Now let us consider the properties of $\nu ^{y}$ satisfying (\ref{v-exp-form}%
). From (\ref{v-exp-form}) we can see that $\forall z_{i}\in W_{i}$,%
\begin{equation}
\log \frac{d\nu _{i}^{y}}{d\mu _{i}}(z_{i})=R_{i}(z_{i}).
\label{distribution_2}
\end{equation}%
Recalling that $\mathbb{E}_{\nu _{i}^{y}}[Z_{i}]=m(\nu _{i}^{y})$, by (\ref%
{v-exp-form}) and (\ref{distribution_2}), we see that%
\begin{equation}
D(\nu _{i}^{y}\parallel \mu _{i})=\int_{W_{i}}\frac{d\nu _{i}^{y}}{d\mu _{i}}%
(z_{i})\log \frac{d\nu _{i}^{y}}{d\mu _{i}}(z_{i})d\mu
_{i}(z_{i})=\int_{W_{i}}R_{i}(z_{i})d\nu _{i}^{y}(z_{i})=R_{i}(m(\nu
_{i}^{y})).  \label{distribution_4}
\end{equation}%
Note that, we did not prove that for any $y\in W$ there exists a measure $%
\nu ^{y}$ satisfying (\ref{v-exp-form}). For each $p\in \mathcal{D}^{\prime
}(\epsilon )$, we construct $\nu ^{p}=(\nu _{1}^{p},\ldots ,\nu _{n}^{p})$
directly at (\ref{d_equation}) below, and show that it satisfies (\ref%
{v-exp-form}), and hence it shares the property (\ref{distribution_4}). For
each $p=p(d)\in \mathcal{D}^{\prime }(\epsilon )$, recalling that $d_{i}$ is
a linear functional from $V_{i}$ to $\mathbb{R}$, we can define $\nu
_{i}^{p} $, a measure on $V_{i}$, as%
\begin{equation}
\frac{d\nu _{i}^{p}}{d\mu _{i}}(z_{i}):=\frac{e^{d_{i}(z_{i})}}{%
\int_{W_{i}}e^{d_{i}(z_{i})}d\mu _{i}(z_{i})}=e^{\lambda
(p_{i})+d_{i}(z_{i})},  \label{d_equation}
\end{equation}%
where $\lambda (p_{i})$ is a normalizing number satisfies that $e^{\lambda
(p_{i})}=(\int_{W_{i}}e^{d_{i}(z_{i})}d\mu _{i}(z_{i}))^{-1}$. From the
construction of $p(d)$ in (\ref{def-p-d}), it is easy to see that%
\begin{equation*}
\int_{W_{i}}z_{i}d\nu _{i}^{p}(z_{i})=\int_{W_{i}}z_{i}e^{\lambda
(p_{i})+d_{i}(z_{i})}d\mu _{i}(z_{i})=p_{i}.
\end{equation*}%
The same approach we used in (\ref{distribution_4}) can be applied here to
show that%
\begin{equation}
D(\nu _{i}^{p}\parallel \mu _{i})=\lambda (p_{i})+d_{i}(p_{i}),
\label{kl_equation}
\end{equation}%
and consequently%
\begin{equation}
\sum_{i=1}^{n}\left( \lambda (p_{i}^{x})+d_{i}^{x}(x_{i})\right)
-\sum_{i=1}^{n}D(\nu _{i}^{p^{x}}\parallel \mu
_{i})=\sum_{i=1}^{n}d_{i}^{x}(x_{i}-p_{i}^{x}).  \label{eq-kl-another-form}
\end{equation}

\subsubsection{\textbf{The approximation (\protect\ref{eq-approximation}) 
\label{section-proof-approximation}}}

Due to (\ref{eq-kl-another-form}), for the approximation (\ref%
{eq-approximation})\textbf{\ }it suffices to bound%
\begin{equation*}
\left\vert f(p^{x})+\sum_{i=1}^{n}d_{i}^{x}(x_{i}-p_{i}^{x})-f(x)\right\vert
\leq \Delta _{1}+\Delta _{2}\text{,}
\end{equation*}%
where%
\begin{eqnarray}
\text{$\Delta _{1}$}&:= &\left\vert f(\widehat{x})-f(x)\right\vert
+\left\vert \sum_{i=1}^{n}f_{i}(x_{\theta _{i}}^{(i)})(\widehat{x}%
_{i}-x_{i})\right\vert ,  \label{def-I} \\
\text{$\Delta _{2}$}&:= &\left\vert f(\widehat{x})-f(p^{x})\right\vert
+\left\vert \sum_{i=1}^{n}d_{i}^{x}(\widehat{x}_{i}-p_{i}^{x})\right\vert
+\left\vert \sum_{i=1}^{n}\left( d_{i}^{x}-f_{i}(x_{\theta
_{i}}^{(i)})\right) (\widehat{x}_{i}-x_{i})\right\vert .  \label{def-II}
\end{eqnarray}%
So the proof of the approximation (\ref{eq-approximation}) consists of the
bounds for $\Delta _{1}$ and $\Delta _{2}$, which will be given separately
below.

\paragraph{\textbf{Bound for} $\Delta _{1}$}

Recall the definition of $\widetilde{\mu }$ (\ref{def-miu-hat}). We show the
following proposition.

\begin{proposition}
\label{Prop-diff-g}Let all notations be as in Theorem \ref{main-theorem}. We
have the following bound%
\begin{equation}
\mathbb{E}_{\widetilde{\mu }}\left[ \left( f(X)-f(\widehat{X})\right) ^{2}%
\right] \leq M^{2}\left(
a\sum_{i=1}^{n}c_{ii}+\sum_{i=1}^{n}b_{i}^{2}\right) +M^{4}\left(
a\sum_{i,j=1}^{n}c_{ij}^{2}+\sum_{i,j=1}^{n}b_{i}b_{j}c_{ij}\right) .
\label{Prop-3-for-f}
\end{equation}
\end{proposition}

\begin{proof}
Let%
\begin{equation*}
h(X):=f(X)-f(\widehat{X}),
\end{equation*}%
and then clearly%
\begin{equation}
\left\vert h(X)\right\vert \leq 2a.  \label{intro-h-a}
\end{equation}%
From the definition of $\widehat{x}$ in (\ref{def-x-hat}), we have%
\begin{equation}
\widehat{\mathrm{x}}_{j}(x)=\frac{\int_{W_{j}}z_{j}e^{f(x_{z_{j}}^{(j)})}d%
\mu _{j}(z_{j})}{\int_{W_{j}}e^{f(x_{z_{j}}^{(j)})}d\mu _{j}(z_{j})}.
\label{eq-x-hat}
\end{equation}%
Note that $\widehat{\mathrm{x}}_{j}(\cdot )$ is a functional from $V$ to $%
V_{j}$. We claim that $\widehat{\mathrm{x}}_{j}(\cdot )$ is Fr\'{e}chet
differentiable (in (\ref{def-f-differentiable}) we just define the notion of
Fr\'{e}chet differentiability for real-valued functional. We can define it
for vector-valued functional similarly, see \cite[chapter 2]{CNVS}). For $%
r\in V$ we let 
\begin{equation*}
r_{0}^{(j)}:=(r_{1},\ldots ,r_{j-1},0,r_{j+1},\ldots ,r_{n}).
\end{equation*}%
Define $\phi _{j}(x)(\cdot ):V\rightarrow V_{j}$ as {\small 
\begin{equation*}
\phi _{j}(x)(r):=\frac{\int_{W_{j}}z_{j}f^{\prime
}(x_{z_{j}}^{(j)})(r_{0}^{(j)})e^{f(x_{z_{j}}^{(j)})}d\mu _{j}(z_{j})}{%
\int_{W_{j}}e^{f(x_{z_{j}}^{(j)})}d\mu _{j}(z_{j})}-\frac{%
\int_{W_{j}}z_{j}e^{f(x_{z_{j}}^{(j)})}d\mu _{j}(z_{j})\int_{W_{j}}f^{\prime
}(x_{z_{j}}^{(j)})(r_{0}^{(j)})e^{f(x_{z_{j}}^{(j)})}d\mu _{j}(z_{j})}{%
\left( \int_{W_{j}}e^{f(x_{z_{j}}^{(j)})}d\mu _{j}(z_{j})\right) ^{2}}.
\end{equation*}%
}By writing out $\widehat{\mathrm{x}}_{j}(x+r)$ and $\widehat{\mathrm{x}}%
_{j}(x)$ according to their definitions and calculating their difference,
due to the fact that $W_{j}$ is compact and $f(\cdot )$ is bounded and Fr%
\'{e}chet differentiable, we can check that 
%it is easy to verify that if we define $\phi _{j}(x)(\cdot
%):V\rightarrow V_{j}$ as
%{\small 
%\begin{equation*}
%\phi _{j}(x)(h):=\frac{\int_{W_{j}}z_{j}f^{\prime }(x_{z_{j}}^{(j)})(h_{0}^{(j)})e^{f(x_{z_{j}}^{(j)})}d\mu
%_{j}(z_{j})}{\int_{W_{j}}e^{f(x_{z_{j}}^{(j)})}d\mu
%_{j}(z_{j})}
%-\frac{\int_{W_{j}}z_{j}e^{f(x_{z_{j}}^{(j)})}d\mu
%_{j}(z_{j})\int_{W_{j}}f^{\prime }(x_{z_{j}}^{(j)})(h_{0}^{(j)})e^{f(x_{z_{j}}^{(j)})}d\mu
%_{j}(z_{j})}{\left( \int_{W_{j}}e^{f(x_{z_{j}}^{(j)})}d\mu
%_{j}(z_{j})\right) ^{2}},
%\end{equation*}%
%}
%{\small 
%\begin{equation*}
%\phi _{j}(x)(h):=\frac{\int_{W_{j}}e^{f(x_{z_{j}}^{(j)})}d\mu
%_{j}(z_{j})\int_{W_{j}}z_{j}f^{\prime }(x_{z_{j}}^{(j)})(h_{0}^{(j)})e^{f(x_{z_{j}}^{(j)})}d\mu
%_{j}(z_{j})+\int_{W_{j}}z_{j}e^{f(x_{z_{j}}^{(j)})}d\mu
%_{j}(z_{j})\int_{W_{j}}f^{\prime }(x_{z_{j}}^{(j)})(h_{0}^{(j)})e^{f(x_{z_{j}}^{(j)})}d\mu
%_{j}(z_{j})}{\left( \int_{W_{j}}e^{f(x_{z_{j}}^{(j)})}d\mu
%_{j}(z_{j})\right) ^{2}},
%\end{equation*}%
%} 
$\widehat{\mathrm{x}}_{j}(x+r)-\widehat{\mathrm{x}}_{j}(x)-\phi
_{j}(x)(r)=o(r)$. We define the partial differential $\frac{d\widehat{%
\mathrm{x}}_{j}(x)}{dx_{i}}(\cdot ):V_{i}\rightarrow V_{j}$ as%
\begin{equation*}
\frac{d\widehat{\mathrm{x}}_{j}(x)}{dx_{i}}(r_{i}):=\phi _{j}((0,\ldots
,r_{i},\ldots ,0)).
\end{equation*}%
Recall the definition of $\widetilde{\mu }$ (\ref{def-miu-hat}). From the
definition of $\phi _{j}(x)(\cdot )$ we can write that for $j\neq i$, 
\begin{eqnarray}
\frac{d\widehat{\mathrm{x}}_{j}(x)}{dx_{i}}(\cdot ) &=&\mathbb{E}_{%
\widetilde{\mu }}[X_{j}f_{i}(X)(\cdot )-\widehat{x}_{j}f_{i}(X)(\cdot )\mid
X_{k}=x_{k}\text{ for }k\neq j]  \notag \\
&=&\mathbb{E}_{\widetilde{\mu }}[(X_{j}-\widehat{x}_{j})(f_{i}(X)-f_{i}(X_{%
\theta _{j}}^{(j)}))(\cdot )\mid X_{k}=x_{k}\text{ for }k\neq j]  \notag \\
&&+\,\,\mathbb{E}_{\widetilde{\mu }}[(X_{j}-\widehat{x}_{j})f_{i}(X_{\theta
_{j}}^{(j)})(\cdot )\mid X_{k}=x_{k}\text{ for }k\neq j].  \label{upper_1}
\end{eqnarray}%
By the definition of $\widehat{x}_{j}$ we have that for any $r\in V$%
\begin{equation}
\mathbb{E}_{\widetilde{\mu }}[(X_{j}-\widehat{x}_{j})f_{i}(X_{\theta
_{j}}^{(j)})(r)\mid X_{k}=x_{k}\text{ for }k\neq j]=0.  \label{upper-bound-8}
\end{equation}%
Due to the fact that%
\begin{equation*}
\left\Vert (f_{i}(X)-f_{i}(X_{\theta _{j}}^{(j)}))(\cdot )\right\Vert
=\left\Vert \int_{0}^{1}f_{ij}(tX+(1-t)X_{\theta _{j}}^{(j)})(\cdot
,X_{j}-\theta _{j})dt\right\Vert \leq c_{ij}M,
\end{equation*}%
we have%
\begin{equation}
\left\Vert \mathbb{E}_{\widetilde{\mu }}[(X_{j}-\widehat{x}%
_{j})(f_{i}(X)-f_{i}(X_{\theta _{j}}^{(j)}))(\cdot )\mid X_{k}=x_{k}\text{
for }k\neq j]\right\Vert \leq c_{ij}M^{2}.  \label{upper-bound-9}
\end{equation}%
Combining (\ref{upper_1}), (\ref{upper-bound-8}) and (\ref{upper-bound-9}),
we see that for $j\neq i$%
\begin{equation}
\left\Vert \frac{d\widehat{\mathrm{x}}_{j}(x)}{dx_{i}}(\cdot )\right\Vert
\leq c_{ij}M^{2}.  \label{bound-derivative-x-hat}
\end{equation}%
Obviously $\frac{d\widehat{\mathrm{x}}_{i}(x)}{dx_{i}}(\cdot )\equiv 0$. For 
$t\in \lbrack 0,1]$ and $x\in W$, we define a linear functional $%
u_{i}(t,x)(\cdot ):V_{i}\rightarrow \mathbb{R}$ as%
\begin{equation}
u_{i}(t,x)(\cdot ):=f_{i}(tx+(1-t)\widehat{x})(\cdot ).  \label{def-u}
\end{equation}%
Then it is clear that%
\begin{equation}
h(x)=\int_{0}^{1}\sum_{i=1}^{n}u_{i}(t,x)(x_{i}-\widehat{x}_{i})dt.
\label{def-h-x}
\end{equation}%
Follow the same idea from \cite[(3.3)]{NLD} to the end of the proof of \cite[%
Lemma 3.1]{NLD}, we can verify that%
\begin{equation}
\left\vert \mathbb{E}_{\widetilde{\mu }}\left[ \left(
u_{i}(t,X)-u_{i}(t,X_{\theta _{i}}^{(i)})\right) (X_{i}-\widehat{X}%
_{i})h(X_{\theta _{i}}^{(i)})\right] \right\vert \leq 2aM\left(
tMc_{ii}+(1-t)M^{3}\sum_{j=1}^{n}c_{ij}^{2}\right) ,  \label{upper_4}
\end{equation}%
and%
\begin{equation}
\left\vert \mathbb{E}_{\widetilde{\mu }}\left[ u_{i}(t,X)(X_{i}-\widehat{X}%
_{i})\left( h(X)-h(X_{\theta _{i}}^{(i)})\right) \right] \right\vert \leq
b_{i}M\left( Mb_{i}+M^{3}\sum_{j=1}^{n}b_{j}c_{ij}\right) .  \label{upper_6}
\end{equation}%
Due to the fact that $\mathbb{E}_{\widetilde{\mu }}\left[ u_{i}(t,X_{\theta
_{i}}^{(i)})(X_{i}-\widehat{X}_{i})h(X_{\theta _{i}}^{(i)})\right] =0$, we
have the following decomposition%
\begin{eqnarray*}
&&\mathbb{E}_{\widetilde{\mu }}\left[ u_{i}(t,X)(X_{i}-\widehat{X}_{i})h(X)%
\right] \\
&=&\mathbb{E}_{\widetilde{\mu }}\left[ \left( u_{i}(t,X)-u_{i}(t,X_{\theta
_{i}}^{(i)})\right) (X_{i}-\widehat{X}_{i})h(X_{\theta _{i}}^{(i)})\right] +%
\mathbb{E}_{\widetilde{\mu }}\left[ u_{i}(t,X)(X_{i}-\widehat{X}_{i})\left(
h(X)-h(X_{\theta _{i}}^{(i)})\right) \right] .
\end{eqnarray*}%
Thus by (\ref{def-h-x}), (\ref{upper_4}) and (\ref{upper_6}), using the
above decomposition we have%
\begin{eqnarray*}
\mathbb{E}_{\widetilde{\mu }}\left[ h^{2}(X)\right] &=&\int_{0}^{1}%
\sum_{i=1}^{n}\mathbb{E}_{\widetilde{\mu }}\left[ u_{i}(t,X)(X_{i}-\widehat{X%
}_{i})h(X)\right] dt \\
&\leq &M^{2}\left( a\sum_{i=1}^{n}c_{ii}+\sum_{i=1}^{n}b_{i}^{2}\right)
+M^{4}\left(
a\sum_{i,j=1}^{n}c_{ij}^{2}+\sum_{i,j=1}^{n}b_{i}b_{j}c_{ij}\right) .
\end{eqnarray*}
\end{proof}

We provide the following proposition, which is also needed for bounding $%
\Delta _{1}$.

\begin{proposition}
\label{Prop-E-G-2}If we denote%
\begin{equation*}
G(x):=\sum_{i=1}^{n}f_{i}(x_{\theta _{i}}^{(i)})(x_{i}-\widehat{x}_{i}),
\end{equation*}%
then%
\begin{equation*}
\mathbb{E}_{\widetilde{\mu }}[G^{2}(X)]\leq
M^{2}\sum_{i=1}^{n}b_{i}^{2}+M^{3}\sum_{i,j=1}^{n}b_{i}(c_{ji}+b_{j}c_{ji}M).
\end{equation*}
\end{proposition}

\begin{proof}
Taking derivative of $G$ and using (\ref{bound-derivative-x-hat}), we have%
\begin{eqnarray}
\left\Vert \frac{\partial G(x)}{\partial x_{i}}(\cdot )\right\Vert
&=&\left\Vert f_{i}(x_{\theta _{i}}^{(i)})(\cdot )+\sum_{j\neq i}^{n}\left(
f_{ji}(x_{\theta _{j}}^{(j)})(x_{j}-\widehat{x}_{j},\cdot )+f_{j}(x_{\theta
_{j}}^{(j)})(-\frac{\partial \widehat{\mathrm{x}}_{j}(x)}{\partial x_{i}}%
(\cdot ))\right) \right\Vert  \notag \\
&\leq &b_{i}+\sum_{j\neq i}^{n}\left( c_{ji}M+b_{j}c_{ji}M^{2}\right) \leq
b_{i}+\sum_{j}^{n}\left( c_{ji}M+b_{j}c_{ji}M^{2}\right) .  \label{upper-7}
\end{eqnarray}%
Following the same idea from \cite[(3.11)]{NLD} to the end of the proof of 
\cite[Lemma 3.2]{NLD}, we finish the proof.
\end{proof}

Next we combine the above two propositions. Denote%
\begin{eqnarray}
B_{1,1}&:= &\left( M^{2}\left(
a\sum_{i=1}^{n}c_{ii}+\sum_{i=1}^{n}b_{i}^{2}\right) +M^{4}\left(
a\sum_{i,j=1}^{n}c_{ij}^{2}+\sum_{i,j=1}^{n}b_{i}b_{j}c_{ij}\right) \right)
^{\frac{1}{2}},  \notag \\
B_{1,2}&:= &\left(
M^{2}\sum_{i=1}^{n}b_{i}^{2}+M^{3}\sum_{i,j=1}^{n}b_{i}(c_{ji}+b_{j}c_{ji}M)%
\right) ^{\frac{1}{2}}.  \label{def-B1B2}
\end{eqnarray}%
And let%
\begin{eqnarray*}
A_{1}&:= &\left\{ x\text{ }\in W\text{, }\left\vert f(x)-f(\widehat{x}%
)\right\vert \leq 2B_{1,1}\right\} , \\
A_{2}&:= &\{x\text{ }\in W\text{, }|\sum_{i=1}^{n}f_{i}(x_{\theta
_{i}}^{(i)})(x_{i}-\widehat{x}_{i})|\leq 2B_{1,2}\}.
\end{eqnarray*}%
Define $A:=A_{1}\cap A_{2}$. Then with Proposition \ref{Prop-diff-g} and
Proposition \ref{Prop-E-G-2} it is easy to see that $\mathbb{P}_{\widetilde{%
\mu }}(A)\geq \frac{1}{2}$. Therefore, with the fact that $%
2(B_{1,1}+B_{1,2})<B_{1}$ (defined in (\ref{def-B})), we have%
\begin{eqnarray}
\log \int_{W}e^{f(x)}d\mu (x) &\leq &\log \int_{A}e^{f(x)}d\mu (x)+\log 2 
\notag \\
&\leq &\log \int_{A}e^{f(\widehat{x})+\sum_{i=1}^{n}f_{i}(x_{\theta
_{i}}^{(i)})(x_{i}-\widehat{x}_{i})}d\mu (x)+B_{1}+\log 2.  \label{final_1}
\end{eqnarray}

\paragraph{\textbf{Bound for} $\Delta _{2}$}

For $\left\vert f(\widehat{x})-f(p^{x})\right\vert $, rewriting it as 
\begin{equation*}
f(\widehat{x})-f(p^{x})=\int_{0}^{1}\sum_{i=1}^{n}f_{i}(t\widehat{x}+(1-t)(%
\widehat{x}-p^{x}))(\widehat{x}_{i}-p_{i}^{x})dt,
\end{equation*}%
by (\ref{lowcomplexity_5}) and Cauchy inequality we have%
\begin{equation}
\left\vert f(\widehat{x})-f(p^{x})\right\vert \leq
\sum_{i=1}^{n}b_{i}\left\Vert \widehat{x}_{i}-p_{i}^{x}\right\Vert
_{V_{i}}\leq \left( \sum_{i=1}^{n}b_{i}^{2}\right) ^{\frac{1}{2}}\left( 
\sqrt{2}M^{3}\left( \sum_{i=1}^{n}c_{ii}^{2}\right) ^{\frac{1}{2}}+\sqrt{2}%
M^{2}n^{\frac{1}{2}}\epsilon \right) .  \label{lowcomplexity_6}
\end{equation}%
For $\left\vert \sum_{i=1}^{n}\left( f_{i}(x_{\theta
_{i}}^{(i)})-d_{i}^{x}\right) (\widehat{x}_{i}-x_{i})\right\vert $, using (%
\ref{covering_condition}) and Cauchy inequality we have%
\begin{equation*}
\left\vert \sum_{i=1}^{n}\left( f_{i}(x)-d_{i}^{x}\right) (\widehat{x}%
_{i}-x_{i})\right\vert \leq M\sqrt{n}\left( \sum_{i=1}^{n}\left\Vert
f_{i}(x)-d_{i}^{x}\right\Vert ^{2}\right) ^{\frac{1}{2}}\leq Mn\epsilon ,
\end{equation*}%
and thus by decomposing $f_{i}(x_{\theta _{i}}^{(i)})-d_{i}^{x}$ as $%
f_{i}(x_{\theta _{i}}^{(i)})-f_{i}(x)$ and $f_{i}(x)-d_{i}^{x}$ we get%
{\small 
\begin{eqnarray}
\left\vert \sum_{i=1}^{n}\left( f_{i}(x_{\theta
_{i}}^{(i)})-d_{i}^{x}\right) (\widehat{x}_{i}-x_{i})\right\vert &\leq
&\left\vert \sum_{i=1}^{n}\left( f_{i}(x_{\theta
_{i}}^{(i)})-f_{i}(x)\right) (\widehat{x}_{i}-x_{i})\right\vert +\left\vert
\sum_{i=1}^{n}\left( f_{i}(x)-d_{i}^{x}\right) (\widehat{x}%
_{i}-x_{i})\right\vert  \notag \\
&\leq &\sum_{i=1}^{n}M^{2}c_{ii}+Mn\epsilon .  \label{lowcomplexity_8}
\end{eqnarray}%
}For the last term $\left\vert \sum_{i=1}^{n}d_{i}^{x}(\widehat{x}%
_{i}-p_{i}^{x})\right\vert $, noting that $\sum_{i=1}^{n}\left\Vert
d_{i}^{x}\right\Vert ^{2}\leq 2\sum_{i=1}^{n}b_{i}^{2}+2\epsilon ^{2}n$, by (%
\ref{lowcomplexity_5}) and Cauchy inequality we have%
\begin{equation}
\left\vert \sum_{i=1}^{n}d_{i}^{x}(\widehat{x}_{i}-p_{i}^{x})\right\vert
\leq \left( 2\sum_{i=1}^{n}b_{i}^{2}+2\epsilon ^{2}n\right) ^{\frac{1}{2}%
}\left( \sqrt{2}M^{3}\left( \sum_{i=1}^{n}c_{ii}^{2}\right) ^{\frac{1}{2}}+%
\sqrt{2}M^{2}n^{\frac{1}{2}}\epsilon \right) .  \label{lowcomplexity_7}
\end{equation}%
Recalling the definition of $\Delta _{2}$, with (\ref{lowcomplexity_6}), (%
\ref{lowcomplexity_8}), (\ref{lowcomplexity_7}) and the definition of $B_{2}$
in (\ref{def-D}), it is clear that%
\begin{equation}
\Delta _{2}\leq B_{2}.  \label{bound-for-II}
\end{equation}

\subsubsection{\textbf{Proof of (\protect\ref{thm-1-upper-bound})}\label%
{Section-upper-bound-final}}

\begin{proof}
By the definition of $\Delta _{2}$ (\ref{def-II}) it is easy to verify that%
\begin{equation*}
\left\vert f(\widehat{x})-f(p^{x})+\sum_{i=1}^{n}f_{i}(x_{\theta
_{i}}^{(i)})(x_{i}-\widehat{x}_{i})-\sum_{i=1}^{n}d_{i}^{x}(x_{i}-p_{i}^{x})%
\right\vert \leq \text{$\Delta _{2}$}.
\end{equation*}%
Define%
\begin{equation*}
C(d):=\{x:x\in W,d^{x}=d\}.
\end{equation*}%
Using (\ref{final_1}) and (\ref{bound-for-II}) we have%
\begin{equation}
\log \int_{W}e^{f(x)}d\mu (x)\leq \log 2+B_{1}+B_{2}+\log \sum_{d\in
D(\epsilon )}\int_{x\in
C(d)}e^{f(p(d))+\sum_{i=1}^{n}d_{i}(x_{i}-p(d)_{i})}d\mu (x).
\label{final-2}
\end{equation}%
From (\ref{d_equation}) it is clear that%
\begin{equation*}
\int_{W}e^{\sum_{i=1}^{n}\lambda (p(d)_{i})+\sum_{i=1}^{n}d_{i}(x_{i})}d\mu
(x)=1.
\end{equation*}%
Combining the above equality and (\ref{kl_equation}), we get the following
bound%
\begin{equation}
\int_{x\in C(d)}e^{f(p(d))+\sum_{i=1}^{n}d_{i}(x_{i}-p(d)_{i})}d\mu (x)\leq
e^{f(p(d))-\sum_{i=1}^{n}d_{i}(p(d)_{i})-\sum_{i=1}^{n}\lambda
(p(d)_{i})}=e^{f(p(d))-\sum_{i=1}^{n}D(\nu _{i}^{p(d)}\parallel \mu _{i})}.
\label{eq-integral-1}
\end{equation}%
Plugging (\ref{eq-integral-1}) into (\ref{final-2}) and noting the fact that
for any $d\in D(\epsilon )$%
\begin{equation*}
f(p(d))-\sum_{i=1}^{n}D(\nu _{i}^{p(d)}\parallel \mu _{i})\leq \max_{\nu \ll
\mu ,\nu =\nu _{1}\times \nu _{2}\times \ldots \times \nu _{n}}\left\{
f(m(\nu ))-\sum_{i=1}^{n}D(\nu _{i}\parallel \mu _{i})\right\} ,
\end{equation*}%
we finish the proof of the upper bound.
\end{proof}

\section{Proofs of applications\label{section-proof-applications}}

In this section we give the proofs of our examples. %we always use $%
%C$ to denote a positive constant independent of the changing variables such
%as $n$.

\subsection{\textbf{Proof of Theorem} \protect\ref{Prop-ex-multi-color}\label%
{section-ex1-proof}}

In this subsection we prove Theorem \ref{Prop-ex-multi-color}. Throughout
the proof, $C$ will denote any positive constant that does not depend on $N$%
. Recall the definitions in Section \ref{ex-multi-color}, and write $n=%
\binom{N}{2}$ for the total number of edges in $G$. Write $\widetilde{T}(x)$
as the normalized version of $T(x)$, that is,%
\begin{equation*}
\widetilde{T}(x):=T(x)/N^{k-2}\text{.}
\end{equation*}%
For $u>1$, by the above definition we see that%
\begin{equation}
T(x)\geq u\mathbb{E[}T(X)]\text{ \ \ }\Longleftrightarrow \text{ \ \ }%
\widetilde{T}(x)\geq tn\text{, \ with }t=\frac{\mathbb{E[}T(X)]}{nN^{k-2}}u.
\label{def-T-tilt}
\end{equation}%
Thanks to the choice of $t$ we have $\psi _{l}(u)=\phi _{l}(t)$, where%
\begin{equation*}
\phi _{l}(t):=\inf \{\sum_{1\leq i<j\leq N}\sum_{s=1}^{l}x_{ijs}\log \frac{%
x_{ijs}}{1/l}:x_{ij}\in W_{0},\text{ }\widetilde{T}((x_{ij})_{1\leq i<j\leq
N})\geq tn\}.
\end{equation*}%
Similarly to the proof of \cite[Theorem 1.1]{NLD}, for $K,\delta >0$ to be
determined later we define%
\begin{equation*}
g(x):=nKh(((\widetilde{T}(x)/n)-t)/\delta ),
\end{equation*}%
where $h(x)=-1$ if $x<-1$, $h(x)=0$ if $x>0$, and for $x\in \lbrack -1,0]$ 
\begin{equation}
h(x)=10(x+1)^{3}-15(x+1)^{4}+6(x+1)^{5}-1.  \label{def-h}
\end{equation}%
By our choice of $h$ we can see that it is negative on $(-1,0)$, with
bounded first and second derivatives. Denote by $\mu _{ij}$ the measure of $%
X_{ij}$ for $1\leq i<j\leq N$, and $\mu $ the measure of $X$. Using the
definition of $g(\cdot )$ we further see that%
\begin{equation}
\mathbb{P}(\widetilde{T}(X)\geq tn)\leq \int_{W_{0}^{n}}e^{g(x)}d\mu (x)%
\text{.}  \label{eq-ex-multi-first-ineq}
\end{equation}%
%
%
%
%
%
%
%where $\mu $ is the product measure of $\mu _{ij}$ for $1\leq i<j\leq N$.
For $s\in \lbrack l]$, let $e_{s}$ be the length $l$ vector with $s$th
coordinate $1$ and other coordinates $0$. Recalling that $\mu _{ij}$ is the
uniform distribution on $\{e_{s},s\in \lbrack l]\}$, we see that for any $%
y_{ij}\in W_{0}$, the only distribution with $\nu _{ij}\ll \mu _{ij}$ and $%
m(\nu _{ij})=y_{ij}$ is $\nu _{ij}(e_{s})=y_{ijs}$ for all $s\in \lbrack l]$%
. Therefore it is easy to see that%
\begin{eqnarray*}
&&\max_{\nu \ll \mu ,\nu =\nu _{1}\times \nu _{2}\times \ldots \times \nu
_{n}}\left\{ g(\left( m(\nu _{ij})\right) _{1\leq i<j\leq N})-\sum_{1\leq
i<j\leq N}D(\nu _{ij}\parallel \mu _{ij})\right\}  \\
&{\tiny =}&\max_{y_{ij}\in W_{0},1\leq i<j\leq N}\left\{ g(\left(
y_{ij}\right) _{1\leq i<j\leq N})-\sum_{1\leq i<j\leq
N}\sum_{s=1}^{l}y_{ijs}\log \frac{y_{ijs}}{1/l}\right\} \text{.}
\end{eqnarray*}%
Let $K=\phi _{l}(t)/n$. We claim that 
%Then same as the proof of \cite[Theorem 1.1]{NLD},
%we get%
\begin{equation}
\max_{y_{ij}\in W_{0},1\leq i<j\leq N}\left\{ g(\left( y_{ij}\right) _{1\leq
i<j\leq N})-\sum_{1\leq i<j\leq N}\sum_{s=1}^{l}y_{ijs}\log \frac{y_{ijs}}{%
1/l}\right\} \leq -\phi _{l}(t-\delta )\text{.}  \label{eq-psi-l-delta}
\end{equation}%
This is because, for $y=\left( y_{ij}\right) _{1\leq i<j\leq N}$, if $%
\widetilde{T}(y)\geq tn$, we have $g(y)=0$, and thus 
\begin{equation*}
g(y)-\sum_{1\leq i<j\leq N}\sum_{s=1}^{l}y_{ijs}\log \frac{y_{ijs}}{1/l}%
=-\sum_{1\leq i<j\leq N}\sum_{s=1}^{l}y_{ijs}\log \frac{y_{ijs}}{1/l}\leq
-\phi _{l}(t)\leq -\phi _{l}(t-\delta ).
\end{equation*}%
If $\widetilde{T}(y)\leq (t-\delta )n$, we have $g(y)=-Kn$, and then 
\begin{equation*}
g(y)-\sum_{1\leq i<j\leq N}\sum_{s=1}^{l}y_{ijs}\log \frac{y_{ijs}}{1/l}\leq
-Kn=-\phi _{l}(t)\leq -\phi _{l}(t-\delta ).
\end{equation*}%
If $\widetilde{T}(y)=(t-\delta ^{\prime })n$ for some $\delta ^{\prime }\in
(0,\delta )$, we have 
\begin{equation*}
g(y)-\sum_{1\leq i<j\leq N}\sum_{s=1}^{l}y_{ijs}\log \frac{y_{ijs}}{1/l}\leq
-\sum_{1\leq i<j\leq N}\sum_{s=1}^{l}y_{ijs}\log \frac{y_{ijs}}{1/l}\leq
-\phi _{l}(t-\delta ^{\prime })\leq -\phi _{l}(t-\delta ).
\end{equation*}%
Observe that if we denote by $\mathcal{D}(\epsilon )$ a $\sqrt{n}\epsilon $%
-covering for the gradient of $\widetilde{T}(x)$ in the sense of (\ref%
{covering_condition}), then $D((\delta \epsilon )/(4K))$ is a $\sqrt{n}%
\epsilon $-covering for the gradient of $g(x)$. Applying Theorem \ref%
{main-theorem} for $g(\cdot )$, with (\ref{eq-ex-multi-first-ineq}) and (\ref%
{eq-psi-l-delta}) we get%
\begin{equation}
\log \mathbb{P}(\widetilde{T}(X)\geq tn)\leq -\phi _{l}(t-\delta
)+B_{1}+B_{2}+\log 2+\log \left\vert D((\delta \epsilon )/(4K))\right\vert .
\label{eq-mid-bound-ex-multicolor}
\end{equation}%
Next we analyze the right-hand side of (\ref{eq-mid-bound-ex-multicolor}).
First we bound $\phi _{l}(t)-\phi _{l}(t-\delta )$.

\subsubsection{\textbf{Upper bound of} $\protect\phi _{l}(t)-\protect\phi %
_{l}(t-\protect\delta )$\label{upper-bound-phi-diff-tech}}

Obviously $\phi _{l}(t)\geq \phi _{l}(t-\delta )$. If $\phi _{l}(t)=\phi
_{l}(t-\delta )$, then $0$ is an upper bound. Now we consider the case that $%
\phi _{l}(t)>\phi _{l}(t-\delta )$, and by the definition of $\phi _{l}$,
the only possibility is that $\phi _{l}(t-\delta )$ is achieved on some $%
x^{\ast }=(x_{ij}^{\ast })_{1\leq i<j\leq N}$ with $x_{ij}^{\ast }\in W_{0}$
and $\widetilde{T}(x^{\ast })\in \lbrack \left( t-\delta \right) n,tn)$.
Note that in addition we can assume $x_{ij1}^{\ast }\geq x_{ij2}^{\ast }\geq
\ldots \geq x_{ijl}^{\ast }$ for all $1\leq i<j\leq N$, since when $%
(\{x_{ij1}^{\ast },x_{ij2}^{\ast },\ldots ,x_{ijl}^{\ast }\})_{1\leq i<j\leq
N}$ is fixed, the choice $x_{ij1}^{\ast }\geq x_{ij2}^{\ast }\geq \ldots
\geq x_{ijl}^{\ast }$ achieves the maximum of $\widetilde{T}(\cdot )$ by the
rearrangement inequality. Thus if this decreasing relation is not satisfied,
we can choose another $x^{\prime }$ satisfying it with $\widetilde{T}%
(x^{\prime })>\widetilde{T}(x^{\ast })$, and $\phi _{l}(t-\delta )$ is
achieved on $x^{\prime }$ too, and we must have $\widetilde{T}(x^{\prime
})\in \lbrack \left( t-\delta \right) n,tn)$ otherwise $\phi _{l}(t)=\phi
_{l}(t-\delta )$.

For any $x=(x_{ij})_{1\leq i<j\leq N}$ with $x_{ij}\in W_{0}$, $\widetilde{T}%
(x)=\left( t-\delta ^{\prime }\right) n$ for some $\delta ^{\prime }\in
\lbrack 0,\delta ]$, and $x_{ij1}\geq x_{ij2}\geq \ldots \geq x_{ijl}$ for
any $1\leq i<j\leq N$, we consider $y=(y_{ij})_{1\leq i<j\leq N}$, where for
some $\gamma >0$ to be determined later,%
\begin{equation*}
y_{ij}:=(1-\gamma )x_{ij}+\gamma e_{1}\text{, \ }e_{1}=(1,0,\ldots ,0).
\end{equation*}%
By the definition of $y$ and $\widetilde{T}$ we have{\small 
\begin{eqnarray}
&&N^{k-2}(\widetilde{T}(y)-\widetilde{T}(x))  \notag \\
&=&\sum_{q_{1},\ldots ,q_{k}\in \left[ N\right] }\big(\prod\limits_{\{r,r^{%
\prime }\}\in E}(x_{q_{r}q_{r^{\prime }}1}+\gamma
\sum_{s=2}^{l}x_{q_{r}q_{r^{\prime }}s})+(1-\gamma
)^{m}\sum_{s=2}^{l}\prod\limits_{\{r,r^{\prime }\}\in E}x_{q_{r}q_{r^{\prime
}}s}-\sum_{s=1}^{l}\prod\limits_{\{r,r^{\prime }\}\in E}x_{q_{r}q_{r^{\prime
}}s}\big).  \label{eq-T-tilt-y-x}
\end{eqnarray}%
} Next we show that 
\begin{equation}
N^{k-2}(\widetilde{T}(y)-\widetilde{T}(x))\geq (m-1)\gamma
^{2}\sum_{q_{1},q_{2},\ldots ,q_{k}\in \left[ N\right] }\left(
\sum_{\{r,r^{\prime }\}\in E}\frac{1-x_{q_{r}q_{r^{\prime }}1}}{%
x_{q_{r}q_{r^{\prime }}1}}\prod\limits_{\{r,r^{\prime }\}\in
E}x_{q_{r}q_{r^{\prime }}1}\right) .  \label{eq-ex1-ad-1}
\end{equation}%
We fix $(q_{1},q_{2},\ldots ,q_{k})\in \left[ N\right] ^{k}$ for our
analysis. Denote by 
\begin{equation*}
I=\prod\limits_{\{r,r^{\prime }\}\in E}(x_{q_{r}q_{r^{\prime }}1}+\gamma
\sum_{s=2}^{l}x_{q_{r}q_{r^{\prime }}s})-(m-1)\gamma ^{2}\left(
\sum_{\{r,r^{\prime }\}\in E}\frac{\sum_{s\geq 2}x_{q_{r}q_{r^{\prime }}s}}{%
x_{q_{r}q_{r^{\prime }}1}}\prod\limits_{\{r,r^{\prime }\}\in
E}x_{q_{r}q_{r^{\prime }}1}\right) .
\end{equation*}%
For each $l^{\prime }\in \lbrack l]$, we let 
\begin{equation*}
\mathcal{M}_{l^{\prime }}:=\{\prod_{\{r,r^{\prime }\}\in
E}x_{q_{r}q_{r^{\prime }}s_{r,r^{\prime }}}:(s_{r,r^{\prime
}})_{\{r,r^{\prime }\}\in E}\in \lbrack l]^{m},max_{\{r,r^{\prime }\}\in E}{%
s_{r,r^{\prime }}}=l^{\prime }\}.
\end{equation*}%
By the decreasing assumption on $x$, we see that each term in $\mathcal{M}%
_{l^{\prime }}$ is greater than or equal to $\prod_{\{r,r^{\prime }\}\in
E}x_{q_{r}q_{r^{\prime }}l^{\prime }}$. By direct calculation, one can check
that, for $2\leq l^{\prime }\leq l$, in the expansion of $%
\prod_{\{r,r^{\prime }\}\in E}(x_{q_{r}q_{r^{\prime }}1}+\gamma
\sum_{s=2}^{l}x_{q_{r}q_{r^{\prime }}s})$, the summation of the coefficients
of those terms in $\mathcal{M}_{l^{\prime }}$ is $g_{0}(l^{\prime })$ where 
\begin{equation*}
g_{0}(l^{\prime }):=(1+(l^{\prime }-1)\gamma )^{m}-(1+(l^{\prime }-2)\gamma
)^{m}.
\end{equation*}%
Similarly, for $2\leq l^{\prime }\leq l$, as $\gamma <m^{-1}$, one can check
that in the expansion of $I$, all the coefficients of terms in $\mathcal{M}%
_{l^{\prime }}$ are positive, and the summation of them is $g_{0}(l^{\prime
})-m(m-1)\gamma ^{2}$. From the above analysis we have that 
\begin{equation}
I\geq \prod\limits_{\{r,r^{\prime }\}\in E}x_{q_{r}q_{r^{\prime
}}1}+\sum_{2\leq s\leq l}\left[ (g_{0}(s)-m(m-1)\gamma
^{2})\prod\limits_{\{r,r^{\prime }\}\in E}x_{q_{r}q_{r^{\prime }}s}\right] .
\label{eq-arrangementineq-1}
\end{equation}%
It is direct to check that $g_{0}(\cdot )$ is increasing on $\mathbb{Z}+$.
Note that we can rewrite 
\begin{equation}
(1-\gamma )^{m}\sum_{s=2}^{l}\prod\limits_{\{r,r^{\prime }\}\in
E}x_{q_{r}q_{r^{\prime }}s}-\sum_{s=1}^{l}\prod\limits_{\{r,r^{\prime }\}\in
E}x_{q_{r}q_{r^{\prime }}s}=-\prod\limits_{\{r,r^{\prime }\}\in
E}x_{q_{r}q_{r^{\prime
}}1}-g_{0}(1)\sum_{s=2}^{l}\prod\limits_{\{r,r^{\prime }\}\in
E}x_{q_{r}q_{r^{\prime }}s}.  \label{eq-arrangementineq-2}
\end{equation}%
Combining (\ref{eq-arrangementineq-1}) and (\ref{eq-arrangementineq-2}), and
using the monotonicity of $g_{0}(\cdot )$, we see that 
\begin{equation*}
I+(1-\gamma )^{m}\sum_{s=2}^{l}\prod\limits_{\{r,r^{\prime }\}\in
E}x_{q_{r}q_{r^{\prime }}s}-\sum_{s=1}^{l}\prod\limits_{\{r,r^{\prime }\}\in
E}x_{q_{r}q_{r^{\prime }}s}\geq (g_{0}(2)-m(m-1)\gamma
^{2}-g_{0}(1))\sum_{2\leq s\leq l}\prod\limits_{\{r,r^{\prime }\}\in
E}x_{q_{r}q_{r^{\prime }}s}\geq 0,
\end{equation*}%
where the last inequality is due to the fact that $g_{0}(2)-m(m-1)\gamma
^{2}-g_{0}(1)\geq 0$. Summing over all possible $(q_{1},q_{2},\ldots
,q_{k})\in \left[ N\right] ^{k}$ leads to (\ref{eq-ex1-ad-1}). \newline
For any $\lambda >1/l$, we denote by $N(\lambda )$ the number of
homomorphisms of $H$ in $G$ whose edges all satisfy $x_{ij1}>\lambda $. Note
that $x_{ij1}\geq 1/l$ always holds since $x_{ij1}\geq x_{ijs}$ for any $%
s\in \lbrack l]$. Denote by $C(N,H)$ the total number of different
homomorphisms of $H$ in a $N$ vertices complete graph. Then we have%
\begin{equation*}
N^{k-2}\widetilde{T}(x)\geq \sum_{q_{1},q_{2},\ldots ,q_{k}\in \left[ N%
\right] }\prod\limits_{\{r,r^{\prime }\}\in E}x_{q_{r}q_{r^{\prime }}1}\geq
N(\lambda )\lambda ^{m}+\left( C(N,H)-N(\lambda )\right) \frac{1}{l^{m}},
\end{equation*}%
which with the fact $\widetilde{T}(x)=\left( t-\delta ^{\prime }\right) n$
implies that%
\begin{equation}
N(\lambda )\leq (\left( t-\delta ^{\prime }\right)
N^{k}/2-C(N,H)/l^{m})/(\lambda ^{m}-1/l^{m}).  \label{eq-bound-N-lambda}
\end{equation}%
We denote by $\Gamma _{1}$ the set of homomorphisms of $H$ in $G$ who have
at least an edge with $x_{ij1}\leq \lambda $. Since $x_{ij1}\leq \lambda $
implies that $\left( 1-x_{ij1}\right) /x_{ij1}\geq \left( 1-\lambda \right)
/\lambda $, we have%
\begin{equation}
\sum_{q_{1},q_{2},\ldots ,q_{k}\in \left[ N\right] }\left(
\sum_{\{r,r^{\prime }\}\in E}\frac{1-x_{q_{r}q_{r^{\prime }}1}}{%
x_{q_{r}q_{r^{\prime }}1}}\prod\limits_{\{r,r^{\prime }\}\in
E}x_{q_{r}q_{r^{\prime }}1}\right) \geq \frac{1-\lambda }{\lambda }%
\sum_{(q_{1},q_{2},\ldots ,q_{k})\in \Gamma _{1}}\prod\limits_{\{r,r^{\prime
}\}\in E}x_{q_{r}q_{r^{\prime }}1},  \label{eq-H-counts-A}
\end{equation}%
where in the right-hand side we use $(q_{1},q_{2},\ldots ,q_{k})\in \Gamma
_{1}$ to represent those $(q_{1},q_{2},\ldots ,q_{k})$ with corresponding
homomorphism $H$ (that is, the homomorphism with vertices $%
(q_{1},q_{2},\ldots ,q_{k})$) in $\Gamma _{1}$. Note that 
%, for convenience. By the definition of $%
%\widetilde{T}(x)$ and $N(\lambda )$ we have%
\begin{equation*}
l\sum_{(q_{1},q_{2},\ldots ,q_{k})\in \Gamma
_{1}}\prod\limits_{\{r,r^{\prime }\}\in E}x_{q_{r}q_{r^{\prime }}1}\geq
\sum_{(q_{1},q_{2},\ldots ,q_{k})\in \Gamma
_{1}}\sum_{s=1}^{l}\prod\limits_{\{r,r^{\prime }\}\in E}x_{q_{r}q_{r^{\prime
}}s}\geq \widetilde{T}(x)N^{k-2}-N(\lambda ).
\end{equation*}%
Combining above inequality and (\ref{eq-bound-N-lambda}), we get{\small 
\begin{equation*}
\sum_{(q_{1},q_{2},\ldots ,q_{k})\in \Gamma _{1}}\prod\limits_{\{r,r^{\prime
}\}\in E}x_{q_{r}q_{r^{\prime }}1}\geq N^{k}\left[ (1-1/N)\left( t-\delta
^{\prime }\right) /2-(\left( t-\delta ^{\prime }\right)
/2-C(N,H)/(N^{k}l^{m}))/(\lambda ^{m}-1/l^{m})\right] /l.
\end{equation*}%
} Due to the fact that $C(N,H)/N^{k}$ converges to a positive constant as $%
N\rightarrow \infty $, and that $t$ is of order $1/l^{m-1}$ by (\ref%
{def-T-tilt}), we can choose $\lambda =1-c/l$ for some constant $c>0$ such
that%
\begin{equation}
\sum_{(q_{1},q_{2},\ldots ,q_{k})\in \Gamma _{1}}\prod\limits_{\{r,r^{\prime
}\}\in E}x_{q_{r}q_{r^{\prime }}1}\geq CN^{k}l^{-(m+1)}.  \label{eq-ex1-ad-2}
\end{equation}%
Combining (\ref{eq-ex1-ad-1}), (\ref{eq-H-counts-A}) and (\ref{eq-ex1-ad-2}%
), we see that $\widetilde{T}(y)-\widetilde{T}(x)\geq C\gamma
^{2}N^{2}l^{-(m+2)}$. Thus if we choose $\gamma =C_{0}\delta
^{1/2}l^{(m+2)/2}$ for a suitable $C_{0}>0$, we have $\widetilde{T}(y)-%
\widetilde{T}(x)\geq \delta n$ and thus $\widetilde{T}(y)\geq tn$. From the
convexity of $x\log x$, we have%
\begin{eqnarray}
\phi _{l}(t) &\leq &\sum_{1\leq i<j\leq N}\sum_{s=1}^{l}y_{ijs}\log \frac{%
y_{ijs}}{1/l}  \notag \\
&\leq &(1-\gamma )\sum_{1\leq i<j\leq N}\sum_{s=1}^{l}x_{ijs}\log \frac{%
x_{ijs}}{1/l}+\gamma n\log l\leq \phi _{l}(t-\delta )+C_{0}N^{2}\delta
^{1/2}l^{(m+2)/2}\log l,  \label{eq-ex-multi-bound-diff-phi}
\end{eqnarray}%
where in the last inequality we let $x=x^{\ast }$.

\subsubsection{{\textbf{Upper bound for} $\protect\phi _{l}(t) $}}

{\small Denote by $\left\lceil t^{1/k}N\right\rceil$ the smallest integer
greater than $t^{1/k}N$. Choose $r=C_{1}\left\lceil t^{1/k}N\right\rceil $,
and let $x=(x_{ij})_{1\leq i<j\leq N}$ where%
\begin{equation*}
x_{ij}=\left\{ 
\begin{array}{cc}
e_{1}, & \text{if }1\leq i<j\leq r, \\ 
(1/l,\ldots ,1/l), & \text{otherwise.}%
\end{array}%
\right.
\end{equation*}%
Then it is easy to check that for a suitable $C_{1}>0$ we have $\widetilde{T}%
(x)\geq tn$ for all $N$. Thus%
\begin{equation}
\phi _{l}(t)\leq \sum_{1\leq i<j\leq N}\sum_{s=1}^{l}x_{ijs}\log \frac{%
x_{ijs}}{1/l}\leq Ct^{2/k}N^{2}\log l.  \label{eq-ex1-ad3}
\end{equation}
}

\subsubsection{{\textbf{Final calculation}}}

We give the proofs of the upper bound and lower bound of Theorem \ref%
{Prop-ex-multi-color} separately below.

\begin{proof}[Proof of the upper bound in Theorem \protect\ref%
{Prop-ex-multi-color}]
Recalling that $K=\phi _{l}(t)/n$, with (\ref{eq-ex1-ad3}) and the fact that 
$t$ is of the order $l^{-(m-1)}$, we can see that%
\begin{equation*}
K\leq Cl^{-\frac{2(m-1)}{k}}\log l\text{.}
\end{equation*}%
We work with the $L_{1}$ norm in this problem. It is easy to verify that for 
$g(x)$ we have%
\begin{equation}
\left\vert g(x)\right\vert \leq nK,\,\,\,\,\left\Vert \frac{\partial g(x)}{%
\partial x_{ij}}\right\Vert \leq \frac{CKb_{ij}^{\prime }}{\delta },
\label{eq-ex-multi-bound-g-gd}
\end{equation}%
\begin{equation}
\left\Vert \frac{\partial ^{2}g(x)}{\partial x_{ij}\partial x_{kl}}%
\right\Vert \leq \frac{CKc_{ij,kl}^{\prime }}{\delta }+\frac{%
CKb_{ij}^{\prime }b_{kl}^{\prime }}{n\delta ^{2}},
\label{eq-ex-multi-bound-g-gd-2}
\end{equation}%
where%
\begin{equation}
b_{i}^{\prime }\leq C\text{, }c_{ij,kl}^{\prime }\leq \left\{ 
\begin{array}{cc}
CN^{-1}, & \text{ if }\left\vert \{i,j,k,l\}\right\vert =2\text{ or }3 \\ 
CN^{-2}, & \text{otherwise}%
\end{array}%
\right. .  \label{eq-ex-multi-bound-b-c}
\end{equation}%
Denoting the $\sqrt{n}\epsilon $-covering set in \cite[Theorem 1.2]{NLD} as $%
\widetilde{D}(\epsilon )$. Since we are working with $L_{1}$ norm and each $%
x_{ij}$ is $l$ dimensional, it is not hard to observe that for any $\epsilon
^{\prime }>0$, $D(\epsilon ^{\prime }):=\widetilde{D}(\epsilon ^{\prime }/%
\sqrt{l})\times \ldots \times \widetilde{D}(\epsilon ^{\prime }/\sqrt{l})$
(the product of $l$ sets) is a $\sqrt{n}\epsilon ^{\prime }$-covering of the
gradient of $\widetilde{T}(x)$ in the sense of (\ref{covering_condition}).
Therefore by \cite[Lemma 5.2]{NLD} we get%
\begin{equation}
\log \left\vert D((\delta \epsilon )/(4K))\right\vert \leq C\frac{l^{3}NK^{4}%
}{\delta ^{4}\epsilon ^{4}}\left( \log N\right) .
\label{eq-ex-multi-bound-cover}
\end{equation}%
Now we bound the right-hand side of (\ref{eq-mid-bound-ex-multicolor}). It
is clear that in this example $M\leq 2$. Using (\ref{eq-ex-multi-bound-g-gd}%
), (\ref{eq-ex-multi-bound-g-gd-2}) and (\ref{eq-ex-multi-bound-b-c}), by
some algebra it is easy to check that under the conditions that%
\begin{equation}
N\delta ^{2}>1,\,\,k/\delta >1,\,\,N\delta \epsilon /K>1,\,\,\epsilon <1,
\label{eq-ex-multi-condition-1}
\end{equation}%
we have%
\begin{equation}
B_{1}=CN^{\frac{3}{2}}K^{\frac{3}{2}}\delta ^{-1},\text{ }%
B_{2}=CN^{2}\epsilon K\delta ^{-1}.  \label{eq-ex-multi-bound-B-D}
\end{equation}%
Thus with (\ref{eq-mid-bound-ex-multicolor}), (\ref%
{eq-ex-multi-bound-diff-phi}), (\ref{eq-ex-multi-bound-cover}) and (\ref%
{eq-ex-multi-bound-B-D}), we see that%
\begin{eqnarray}
\log \mathbb{P}(\widetilde{T}(X)\geq tn) &\leq &-\phi _{l}(t)+\log
2+CN^{2}\delta ^{1/2}l^{(m+2)/2}\log l+CN^{\frac{3}{2}}K^{\frac{3}{2}}\delta
^{-1}  \notag \\
&&+\,\,CN^{2}\epsilon K\delta ^{-1}+C\frac{l^{3}NK^{4}}{\delta ^{4}\epsilon
^{4}}\left( \log N\right) .  \label{eq-ex-multi-mid2-bound}
\end{eqnarray}%
Denote by $T_{i}(X)$ the number of homomorphisms of $H$ in $G$ whose edges
are all of color $i$, and let $\widetilde{T}_{i}(X):=T_{i}(X)/N^{k-2}$. Then
obviously $T_{i}(X)$ has the same distribution as the number of
homomorphisms of $H$ in $G(N,l^{-1})$ - the Erd\H{o}s-R\'{e}nyi random graph
with probability $l^{-1}$ , and thus by \cite[Theorem 1.2 and Theorem 1.5]%
{SKR04} we have%
\begin{equation}
-\log \mathbb{P}(\widetilde{T}_{i}(X)\geq tl^{-1}n)\geq CN^{2}l^{-\Delta }.
\label{eq-ex-multi-lowerbound-T-i}
\end{equation}%
Due to the fact that $\mathbb{P}(\widetilde{T}(X)\geq tn)=\mathbb{P}%
(\sum_{i=1}^{l}\widetilde{T}_{i}(X)\geq tn)\leq \sum_{i=1}^{l}\mathbb{P}(%
\widetilde{T}_{i}(X)\geq tl^{-1}n)$, from (\ref{eq-ex-multi-lowerbound-T-i})
we further get%
\begin{equation}
-\log \mathbb{P}(\widetilde{T}(X)\geq tn)\geq -\log l\mathbb{P}(\widetilde{T}%
_{i}(X)\geq tl^{-1}n)\geq CN^{2}l^{-\Delta }-\log l.
\label{eq-ex-multi-lower-log-T}
\end{equation}%
Choosing $\epsilon =N^{-1/5}\delta ^{-3/5}K^{3/5}l^{3/5}(\log N)^{1/5}$, $%
\delta =N^{-(2\Delta +m+2)/(19+8m+21\Delta )}(\log N)^{-4}$, with (\ref%
{eq-ex-multi-lowerbound-T-i}) and (\ref{eq-ex-multi-lower-log-T}) it is
directly to derive that{\small 
\begin{eqnarray*}
\frac{\phi _{l}(t)}{-\log \mathbb{P}(\widetilde{T}\geq tn)} &\leq &1+CN^{-%
\frac{m/2+\Delta +1}{19+8m+21\Delta }}l^{\frac{m+2\Delta +2}{2}}(\log
l)(\log N)^{-2}+CN^{\frac{m+2\Delta +2}{19+8m+21\Delta }-\frac{1}{2}}l^{-%
\frac{3(m-1)-\Delta k}{k}}(\log l)^{\frac{3}{2}}(\log N)^{4} \\
&&+\,\,CN^{-\frac{1}{5}+\frac{8}{5}\frac{m+2\Delta +2}{19+8m+21\Delta }}l^{%
\frac{3}{5}-\frac{16(m-1)}{5k}+\Delta }(\log l)^{\frac{8}{5}}(\log N)^{\frac{%
33}{5}}+o(1).
\end{eqnarray*}%
} \iffalse%
\begin{eqnarray*}
\frac{\phi _{l}(t)}{-\log \mathbb{P}(\widetilde{T}\geq tn)} &\leq &CN^{-%
\frac{m/2+\Delta +1}{19+8m+21\Delta }}l^{m/2+\Delta +1}(\log l)(\log
N)^{-2}+CN^{-\frac{1}{2}+\frac{m+2\Delta +2}{19+8m+21\Delta }}l^{-\frac{%
3(m-1)}{k}+\Delta }(\log l)^{3/2}(\log N)^{4} \\
&&+CN^{-\frac{1}{5}+\frac{8}{5}\frac{m+2\Delta +2}{19+8m+21\Delta }}l^{\frac{%
3}{5}-\frac{16(m-1)}{5k}+\Delta }(\log l)^{\frac{8}{5}}(\log N)^{\frac{33}{5}%
}+1+o(1).
\end{eqnarray*}%
\fi Using above equation and the fact that $(m-1)/k<\Delta /2$, we can check
that if $l\leq N^{1/(19+8m+21\Delta )}$, then the right-hand side goes to $0$
as $N\rightarrow \infty $, and it is directly to verify that condition (\ref%
{eq-ex-multi-condition-1}) holds. Recalling that $\psi _{l}(u)=\phi _{l}(t)$%
, we finish the proof.
\end{proof}

{\small \bigskip }

Next we show the lower bound.

\begin{proof}[Proof of the lower bound in Theorem \protect\ref%
{Prop-ex-multi-color}]
Fix any $z\in W_{0}^{n}$ such that $\widetilde{T}(z)\geq (t+\delta _{0})n$,
with $\delta _{0}$ to be determined later. Recall that $e_{s}$ is the $l$%
-dimension vector with $1$ on the $s$th coordinate and $0$ on others. Let $%
Z_{ij}$, $1\leq i<j\leq N$ be independent random vectors with $\mathbb{P}%
(Z_{ij}=e_{s})=z_{ijs},$ and denote by $\widehat{\mu }$ the measure of $%
Z=(Z_{ij})_{1\leq i<j\leq N}$. Let $\Gamma $ be the set of $x\in W_{0}^{n}$
such that $\widetilde{T}(x)\geq tn$, and let $\Gamma ^{\prime }$ be the
subset of $\Gamma $ where%
\begin{equation*}
\left\vert \sum_{1\leq i<j\leq N}\sum_{k=1}^{l}\left( x_{ijk}\log
z_{ijk}-x_{ijk}\log \frac{1}{l}-z_{ijk}\log \left( lz_{ijk}\right) \right)
\right\vert \leq \epsilon _{0}n.
\end{equation*}%
Then we have%
\begin{eqnarray}
\mathbb{P}(\widetilde{T}(X)\geq tn)=\int_{\Gamma }1d\mu (x) &\geq
&\int_{\Gamma ^{\prime }}e^{\sum_{1\leq i<j\leq N}\sum_{k=1}^{l}\left(
-x_{ijk}\log z_{ijk}+x_{ijk}\log z_{ijk}\right) }d\mu (x)  \notag \\
&\geq &e^{z_{ijk}\log \left( lz_{ijk}\right) -\epsilon _{0}n}\mathbb{P}_{%
\widehat{\mu }}(Z\in \Gamma ^{\prime }).  \label{Jun-add-319-t-1}
\end{eqnarray}%
Denote%
\begin{equation*}
H(x):=\sum_{1\leq i<j\leq N}\sum_{k=1}^{l}\left( x_{ijk}\log \left(
lz_{ijk}\right) -z_{ijk}\log \left( lz_{ijk}\right) \right) .
\end{equation*}%
Obviously $\mathbb{E}_{\widehat{\mu }}\left[ H(Z)\right] =0$. By direct
calculation we have%
\begin{equation}
\text{Var}_{\widehat{\mu }}(H(Z))=\sum_{1\leq i<j\leq N}\left[
\sum_{k=1}^{l}z_{ijk}\log \left( z_{ijk}\right) -\left(
\sum_{k=1}^{l}z_{ijk}\log \left( z_{ijk}\right) \right) ^{2}\right] .
\label{Var-miu-H}
\end{equation}%
Noting that%
\begin{equation*}
\log \left( \frac{1}{l}\right) \leq \sum_{k=1}^{l}z_{ijk}\log \left(
z_{ijk}\right) \leq 0,
\end{equation*}%
with (\ref{Var-miu-H}) it is clear that%
\begin{equation*}
\text{Var}_{\widehat{\mu }}(H(Z))\leq Cn\left( \log l\right) ^{2}.
\end{equation*}%
Thus by choosing $\epsilon _{0}=C_{2}N^{-1}\left( \log l\right) $ for a
suitable $C_{2}>0$ we have%
\begin{equation}
\mathbb{P}_{\widehat{\mu }}(\left\vert H(Z)\right\vert >\epsilon _{0}n)\leq 
\frac{C\left( \log l\right) ^{2}}{\epsilon _{0}^{2}n}=\frac{1}{4}.
\label{eq-ex1-ad-4}
\end{equation}%
Let $S(x):=\widetilde{T}(x)-\widetilde{T}(z)$. Using the similar approach as
in \cite[(4.3) - (4.4)]{NLD} we can verify that%
\begin{equation*}
\mathbb{E}_{\widehat{\mu }}\left[ S^{2}\right] \leq CN^{2}\text{.}
\end{equation*}%
Thus by choosing $\delta _{0}=C_{3}N^{-1}$ for a suitable $C_{3}>0$, we get%
\begin{equation}
\mathbb{P}_{\widehat{\mu }}(\widetilde{T}(Z)\leq tn)\leq \frac{CN^{2}}{%
\delta _{0}^{2}n^{2}}=\frac{1}{4}.  \label{eq-ex1-ad-5}
\end{equation}%
Using (\ref{eq-ex1-ad-4}) and (\ref{eq-ex1-ad-5}) we see that $\mathbb{P}_{%
\widehat{\mu }}(Z\in \Gamma ^{\prime })\geq 1/2$, therefore with (\ref%
{Jun-add-319-t-1}) and by taking the sup over $z$ we get%
\begin{eqnarray*}
\log \mathbb{P}(\widetilde{T}(X)\geq tn) &\geq &-\phi _{l}(t+\delta
_{0})-\epsilon _{0}n-\log 2 \\
&\geq &-\phi _{l}(t)-CN^{2}\delta _{0}^{1/2}l^{(m+2)/2}\log l-CN\left( \log
l\right) -\log 2.
\end{eqnarray*}%
Consequently with (\ref{eq-ex-multi-lower-log-T}) we see that%
\begin{equation*}
\frac{-\phi _{l}(t)}{-\log \mathbb{P}(\widetilde{T}(X)\geq tn)}\geq 1-CN^{-%
\frac{1}{2}}l^{\Delta +\frac{m+2}{2}}\log N+o(1),
\end{equation*}%
which completes the proof.
\end{proof}

\subsection{{\textbf{Proof of Theorem \protect\ref{Prop-ex-continuous}}\label%
{section-ex2-proof}}}

In this subsection we show Theorem \ref{Prop-ex-continuous} in our second
example about continuous weighted triangle counts. Throughout the proof, $C$
will denote any positive constant that does not depend on $N$. We follow the
routine of the above example. In the proof we use the definitions in Section %
\ref{ex-continuous}. Define the normalized weighted triangle counts $%
\widetilde{T}(x)$ as%
\begin{equation*}
\widetilde{T}(x):=T(x)/N.
\end{equation*}%
For any $1<u<8$, we let $t=u(N-2)/(24N)$. Since $n=N(N-1)/2$ and by
calculation $\mathbb{E}[T(X)]=N(N-1)(N-2)/48$, we see that $\{T(x)\geq u$$%
\mathbb{E}\mathbb{[}T(X)]\}=$ $\{\widetilde{T}(x)\geq tn\}$. Define%
\begin{equation*}
\phi _{n}(t):=\inf \{\sum_{1 \leq i<j \leq N}(-1+\frac{\lambda
_{y_{ij}}e^{-\lambda _{y_{ij}}}}{1-e^{-\lambda _{y_{ij}}}}+\log (\frac{%
\lambda _{y_{ij}}}{1-e^{-\lambda _{y_{ij}}}})):y_{ij}\in (0,1)\text{ such
that }\widetilde{T}(y)\geq tn\}.
\end{equation*}%
Obviously $\phi _{n}(t)=\psi _{n}(u)$. Let $g(x)=nKh(((\widetilde{T}%
(x)/n)-t)/\delta )$ for $h(\cdot)$ defined in (\ref{def-h}), with $K=$ $\phi
_{n}(t)/n$ and $\delta $ to be determined later. Then same as the argument
of showing (\ref{eq-psi-l-delta}), we have%
\begin{equation*}
\max_{y=(y_{ij})_{1\leq i<j\leq N}\text{, }y_{ij}\in (0,1)}\left\{
g(y)-\sum_{i<j}D(\nu ^{y_{ij}}\parallel \mu _{ij})\right\} \leq -\phi
_{n}(t-\delta ).
\end{equation*}%
Applying Theorem \ref{main-theorem} for $g(x)$ and some $\epsilon$ to be
determined later, we get%
\begin{equation}
\log \mathbb{P}(\widetilde{T}(X)\geq tn)\leq \log \mathbb{E[}e^{g(x)}]\leq
-\phi _{n}(t-\delta )+\log 2+B_{1}+B_{2}+\log \left\vert \mathcal{D}%
(\epsilon )\right\vert,  \label{eq-ex-continuous-ad-1}
\end{equation}%
where $B_{1},B_{2}$ are as defined in Theorem \ref{main-theorem}, and $%
\mathcal{D}(\epsilon )$ will be constructed later. Next we upper bound the
rightmost side of (\ref{eq-ex-continuous-ad-1}).

\subsubsection{{\textbf{The upper bound for} $\protect\phi _{n}(t)-\protect%
\phi _{n}(t-\protect\delta )$}}

%{\small 
Recall the definition of $\nu ^{a}$ in Section \ref{ex-continuous}. For $%
\lambda ^{a}>0$ we define $f_{1}(\lambda ^{a}):=$$\mathbb{E}_{\nu ^{a}}[X]$.
After calculation we have%
\begin{equation*}
f_{1}(x)=\frac{1}{x}-\frac{1}{e^{x}-1},
\end{equation*}%
and we can check that on any bounded interval $[-M_{0},M_{0}]$, there exists 
$c_{M_{0}}>0$ such that%
\begin{equation}
f_{1}^{\prime }(x)<-c_{M_{0}}.
\label{condition-ex-continuous-derivative-mean}
\end{equation}%
For $\lambda ^{a}>0$ we define $f_{2}(\lambda ^{a}):=D(\nu ^{a}||U)$, which
after some calculation is 
\begin{equation*}
f_{2}(x)=-1+\frac{xe^{-x}}{1-e^{-x}}+\log (\frac{x}{1-e^{-x}}).
\end{equation*}%
We can check that%
\begin{equation}
f_{2}^{\prime }(x)<0\text{ when }x<0\text{; }f_{2}^{\prime }(x)>0\text{ when 
}x>0\text{; }\left\vert f_{2}^{\prime }(x)\right\vert \leq C_{D}\text{ for
some }C_{D}<\infty .  \label{condition-ex-continuous-derivative-KL}
\end{equation}%
We assume that $t-\delta >1/24$, since later we will choose $\delta
\rightarrow 0$ as $N\rightarrow 0$, and by our choice $t>1/24$ as $%
N\rightarrow 0$. In order to bound $\phi _{n}(t)-\phi _{n}(t-\delta )$, we
use the same strategy as Section \ref{upper-bound-phi-diff-tech}. If $\phi
_{n}(t)\neq \phi _{n}(t-\delta )$, we assume that $\phi _{n}(t-\delta )$ is
achieved on some $z=(z_{ij})_{1\leq i<j\leq N}$ such that $\widetilde{T}%
(z)=\left( t-\delta ^{\prime }\right) n$ for some $\delta ^{\prime }\in
\lbrack 0,\delta ]$. In addition we assume that $z_{ij}\geq 1/2$ for all $i<j
$, since otherwise according to (\ref{condition-ex-continuous-derivative-KL}%
) we can change those $z_{ij}<1/2$ to $1/2$ without increasing $%
\sum_{i<j}D(\nu ^{z_{ij}}||U)$, which results in a bigger $\widetilde{T}(z)$%
, and we can consider the new $z$ instead. For some $s\in (1/2,1)$ to be
determined later, we define $A(s):=\{\{i,j\}:z_{ij}\geq s\}$ and $%
V_{s}(i):=|\{k\in \lbrack N]:z_{ik}\geq s\}|$ (here $|\cdot |$ refers to
cardinality). Write $B(s)$ as the set of triangles whose three edges all
belong to $A(s)$. Observing that for each edge $\{i,j\}\in A(s)$, the number
of triangles in $B(s)$ containing $\{i,j\}$ is at least $%
V_{s}(i)+V_{s}(j)-N-1$, we get that%
\begin{eqnarray}
\left\vert B(s)\right\vert  &\geq &\frac{1}{3}\sum_{\{i,j\}\in
A(s)}(V_{s}(i)+V_{s}(j)-N-1)  \notag \\
&=&\frac{1}{3}\left( \sum_{i=1}^{N}\left( V_{s}(i)\right) ^{2}-\left\vert
A(s)\right\vert (N-1)\right) \geq \frac{1}{3}\left( \frac{4\left\vert
A(s)\right\vert ^{2}}{N}-\left\vert A(s)\right\vert (N-1)\right) ,
\label{eq-ex-cont-Bs-lower}
\end{eqnarray}%
where the second equality is by the fact that each $V_{s}(i)$ appears $%
V_{s}(i)$ times in the summation, and the last inequality is by Cauchy
inequality and the fact that $\sum_{i=1}^{N}\left\vert V_{s}(i)\right\vert
=2\left\vert A(s)\right\vert $. Since%
\begin{equation}
\binom{N}{3}\geq T(z)\geq \left\vert B(s)\right\vert (s^{3}-1/8)+\mathbb{E[}%
T(X)],  \label{eq-ex-cont-Bs-lower-2}
\end{equation}%
with the fact that $\mathbb{E}\mathbb{[}T(X)]=N^{3}/48+o(N^{2})$,
substituting (\ref{eq-ex-cont-Bs-lower}) into (\ref{eq-ex-cont-Bs-lower-2})
we can verify that there exist $s\in (0,1)$ and $c_{s}>0$ independent of $N$
such that $\left\vert A(s)\right\vert \leq (1-c_{s})n$. We find $c_{s}n$
number of edges in $A(s)^{C}$, and increase the weights on them by $\sigma >0
$ to be determined later, getting a new weight vector $\widetilde{z}=(%
\widetilde{z}_{ij})_{1\leq i<j\leq N}$. Later we can verify that $\sigma
\rightarrow 0$ as $N\rightarrow \infty $, and thus the weight-increasing
operation is feasible, that is, $\widetilde{z}_{ij}\leq 1$ for all $\{i,j\}$%
, as $N$ is large enough. Since for each edge there are $N-2$ triangles
containing it, after the operation, with the fact that $z_{ij}>1/2$ for all $%
\{i,j\}$, each edge in $A(s)^{C}$ at least contribute $\sigma /5$ more to $%
\widetilde{T}(z)$. Therefore we get%
\begin{equation*}
\widetilde{T}(\widetilde{z})-\widetilde{T}(z)\geq \frac{c_{s}n\sigma }{5},
\end{equation*}%
which implies that we can choose $\sigma =c_{s}^{\prime }\delta ^{\prime }$
for some $c_{s}^{\prime }>0$ to make $\widetilde{T}(\widetilde{z})\geq tn$.
Since for $N$ large enough we can find $s_{1}<1$ such that $s+\sigma <s_{1}$%
, with (\ref{condition-ex-continuous-derivative-mean}), we see that for
those $z_{ij}\in A(s)^{C}$, we have $\left\vert \lambda ^{z_{ij}}-\lambda ^{%
\widetilde{z}_{ij}}\right\vert \leq c_{s_{1}}\sigma $ for some $c_{s_{1}}>0$%
, and thus with (\ref{condition-ex-continuous-derivative-KL}) we have $D(\nu
^{\widetilde{z}_{ij}}||U)-D(\nu ^{z_{ij}}||U)\leq C_{D}c_{s_{1}}\sigma $.
Therefore we have that%
\begin{equation}
\phi _{n}(t)-\phi _{n}(t-\delta )\leq \sum_{i<j}D(\nu ^{\widetilde{z}%
_{ij}}\parallel \mu _{ij})-\sum_{i<j}D(\nu ^{z_{ij}}\parallel \mu _{ij})\leq
C_{D}c_{s_{1}}\sigma n\leq C_{D}c_{s_{1}}c_{s}^{\prime }\delta ^{\prime
}n\leq CN^{2}\delta .  \label{upper-bound-difference-phi}
\end{equation}%
%
%
%
%
%Recalling that $\delta ^{\prime }\leq \delta $ and taking the inf over $z$
%with $\widetilde{T}(z)\in \lbrack (t-\delta )n,tn)$, we see that there
%exists $C_{t}>0$ such that%
%\begin{equation}
%\phi _{n}(t)-\phi _{n}(t-\delta )\leq C_{t}N^{2}\delta .
%\label{upper-bound-difference-phi}
%\end{equation}
%}

\subsubsection{{\textbf{Bound for} $K$}}

In order to bound $K$, we just need to bound $\phi _{n}(t)$. Obviously we
can choose $z_{ij}=s_{t}$ for some $s_{t}\in (0,1)$ such that $\widetilde{T}%
(z)\geq tn$ for all $n$, and thus $\phi _{n}(t)\leq CN^{2}$, which implies
that $K\leq C$ since $K=\phi _{n}(t)/n$.

\subsubsection{{\textbf{Final calculation}}}

We give the proofs of the upper bound and lower bound of Theorem \ref%
{Prop-ex-continuous} separately below.

\begin{proof}[Proof of the upper bound in Theorem \protect\ref%
{Prop-ex-continuous}]
From our choice of $g$, it is easy to verify that%
\begin{equation}
B_{1}=CN^{3/2}\delta ^{-1}+CN\delta ^{-2}\text{, }\,\,B_{2}=CN\delta
^{-2}+N^{2}\delta ^{-1}\epsilon .  \label{eq-ex-continuous-B-D}
\end{equation}%
One can check that in the sense of (\ref{covering_condition}), the $\sqrt{n}%
\delta \epsilon /(4K)$-covering of the gradient of $\widetilde{T}(x)$ is a $%
\sqrt{n}\epsilon $-covering of the gradient of $g(x)$, by \cite[Lemma 5.2]%
{NLD} and the fact that $K$ is bounded by a constant, we have that for $g(x)$%
, $\log \left\vert D(\epsilon )\right\vert \leq CN\delta ^{-4}\epsilon
^{-4}\log N$. Choosing $\epsilon =N^{-1/5}\delta ^{2/5}$, by (\ref%
{eq-ex-continuous-ad-1}), (\ref{upper-bound-difference-phi}) and (\ref%
{eq-ex-continuous-B-D}) we get%
\begin{equation}
\log \mathbb{P}(\widetilde{T}(X)\geq tn)\leq -\phi _{n}(t)+CN^{2}\delta +C%
\frac{N^{\frac{3}{2}}}{\delta }+C\frac{N}{\delta ^{2}}+CN^{\frac{9}{5}%
}\delta ^{-\frac{8}{5}}\left( \log N\right) ^{\frac{1}{5}}\text{.}
\label{eq-ex-cont-upper-final}
\end{equation}%
For any $s^{\ast }\in (0,1)$, based on the graph $G$ and weight $X$, we
construct a graph $G_{s^{\ast }}^{\prime }(X)$ by making those edges with
weight $>s^{\ast }$ as connected and other edges as disconnected. Write $%
T_{s^{\ast }}(X)$ as the number of triangles in $G_{s^{\ast }}^{\prime }(X)$%
. Then it is not hard to see that we can choose $0<s_{u}<1$ and $1<u^{\prime
}<8$ such that%
\begin{equation*}
\{T(X)\geq u\mathbb{E}\left[ T(X)\right] \}\subset \{T_{s_{u}}(X)\geq
u^{\prime }\mathbb{E}\left[ T_{s_{u}}(X)\right] \}.
\end{equation*}%
Since $G_{s_{u}}^{\prime }(X)$ is just the Erd\H{o}s-R\'{e}nyi random graph $%
G(N,1-s_{u})$, with \cite[Theorem 1.2 and Theorem 1.5]{SKR04} we see that%
\begin{equation}
-\log \mathbb{P}(\widetilde{T}(X)\geq tn)\geq -\log \mathbb{P}%
(T_{s_{u}}(X)\geq u^{\prime }\mathbb{E}\left[ T_{s_{u}}(X)\right] )\geq
CN^{2}.  \label{eq-ex-continuous-ad-4}
\end{equation}%
Choosing $\delta =N^{-1/10}$ and dividing both sides of (\ref%
{eq-ex-cont-upper-final}) by $-\log \mathbb{P}(\widetilde{T}(X)\geq tn)$, we
get the desired upper bound.
\end{proof}

\begin{proof}[Proof of the lower bound in Theorem \protect\ref%
{Prop-ex-continuous}]
Fix any $z=(z_{ij})_{1\leq i<j\leq N}$ with $z_{ij}\in (0,1)$ and $%
\widetilde{T}(z)\geq (t+\delta _{0})n$ with $\delta _{0}$ to be determined
later. Consider $Z=(Z_{ij})_{1\leq i<j\leq N}$ with $Z_{ij}$ ($i<j$)
independently from $\nu ^{z_{ij}}$. Denote by $\widehat{\mu }_{z}$ the
distribution of $Z$. Denote%
\begin{equation*}
\Gamma :=\{x=(x_{ij})_{1\leq i<j\leq N}:x_{ij}\in (0,1),\widetilde{T}(x)\geq
tn\},
\end{equation*}%
and%
\begin{equation*}
\Gamma ^{\prime }:=\Gamma \cap \{x=(x_{ij})_{1\leq i<j\leq N}:\
|\sum_{i<j}(-\lambda _{z_{ij}}x_{ij}-(-1+\frac{\lambda _{z_{ij}}e^{-\lambda
_{z_{ij}}}}{1-e^{-\lambda _{z_{ij}}}}))|<\epsilon _{0}n\},
\end{equation*}%
for $\epsilon _{0}$ to be determined later. Noting that $\mathbb{P}(%
\widetilde{T}(X)\geq tn)=$ $\mathbb{E}\left[ 1_{\Gamma }\right] $ and $%
\Gamma ^{\prime }\subset \Gamma $, we have%
\begin{eqnarray}
\mathbb{P}(\widetilde{T}(X)\geq tn) &\geq &\mathbb{E[}1_{\Gamma ^{\prime
}}e^{\sum_{i<j}(-\lambda _{z_{ij}}x_{ij}+\log (\frac{\lambda _{z_{ij}}}{%
1-e^{-\lambda _{z_{ij}}}}))-\sum_{i<j}(-\lambda _{z_{ij}}x_{ij}+\log (\frac{%
\lambda _{z_{ij}}}{1-e^{-\lambda _{z_{ij}}}}))}]  \notag \\
&\geq &e^{-\sum_{i<j}D(\nu ^{z_{ij}}||U)-\epsilon _{0}n}\mathbb{P}_{\widehat{%
\mu }_{z}}(Z\in \Gamma ^{\prime }).  \label{eq-ex-continuous-ad-3}
\end{eqnarray}%
By direct integration, we can see that for some $C_{4}>0$ 
\begin{equation*}
\mathbb{E}_{\widehat{\mu }_{z}}[(\sum_{i<j}(-\lambda _{z_{ij}}Z_{ij}-(-1+%
\frac{\lambda _{z_{ij}}e^{-\lambda _{z_{ij}}}}{1-e^{-\lambda _{z_{ij}}}}%
)))^{2}]\leq C_{4}n.
\end{equation*}%
%for some $C>0$. 
Thus by choosing $\epsilon _{0}=(4C_{4}/n)^{1/2}$ and using the Markov's
inequality, we get%
\begin{equation*}
\mathbb{P}_{\widehat{\mu }_{z}}(\sum_{i<j}(-\lambda _{z_{ij}}Z_{ij}-(-1+%
\frac{\lambda _{z_{ij}}e^{-\lambda _{z_{ij}}}}{1-e^{-\lambda _{z_{ij}}}}%
))\geq \epsilon _{0}n)\leq \frac{C_{4}}{\epsilon _{0}^{2}n}=\frac{1}{4}.
\end{equation*}%
Using the similar method as in \cite[(4.4)]{NLD}, by choosing $\delta
_{0}=CN^{-1}$, we have that $\mathbb{P}_{\widehat{\mu }_{z}}(\widetilde{T}%
(Z)\leq tn)\leq 1/4$. Thus $\mathbb{P}_{\widehat{\mu }_{z}}(Z\in \Gamma
^{\prime })\geq 1/2$, and with (\ref{eq-ex-continuous-ad-3}) by taking sup
over $z$ we get%
\begin{equation*}
\log \mathbb{P}(\widetilde{T}(X)\geq tn)\geq -\phi _{n}(t+\delta
_{0})-\epsilon _{0}n-\log 2.
\end{equation*}%
Combining above inequality and (\ref{upper-bound-difference-phi}) we get%
\begin{equation*}
\log \mathbb{P}(\widetilde{T}(X)\geq tn)\geq -\phi _{n}(t)+CN-\log 2,
\end{equation*}%
which implies the lower bound with (\ref{eq-ex-continuous-ad-4}).
\end{proof}

\subsection{{\textbf{Proof of Theorem \protect\ref{theorem-as-new}}\label%
{section-new-ex-proof}}}

In this section we show Theorem \ref{theorem-as-new}, which is an extension
of \cite{AS}. Throughout the proof, $C$ will denote any positive constant
that does not depend on $n$. Note that in this example $N$ is the dimension
of $W_{1}$, and it has no relation with $n$. Recall the definitions in
Section \ref{ex-3}. For convenience we define%
\begin{equation*}
f(x):=H_{n}^{J,h}(x)=\frac{1}{2}\sum_{i,j=1}^{n}A_{n}(i,j)x_{i}^{T}Jx_{j}+%
\sum_{i=1}^{n}x_{i}^{T}h,
\end{equation*}%
and%
\begin{equation*}
\widetilde{f}(x):=\frac{1}{2}\sum_{i,j=1}^{n}A_{n}(i,j)x_{i}^{T}Jx_{j}.
\end{equation*}%
Without the loss of generality we assume that $\mu _{i}$'s are supported on
the unit ball $B_{\mathbb{R}^{N}}(1)$ in $\mathbb{R}^{N}$. Using the similar
argument to \cite[Lemma 3.1]{AS}, we can further assume that%
\begin{equation}
\max_{i,j}\left\vert A_{n}(i,j)\right\vert =o(1)\text{ \ \ \ and \ \ }%
A_{n}(i,i)=0\text{ for all }i.  \label{eq-new-extra-condition}
\end{equation}%
We work with the $L_{1}$ norm, and note that%
\begin{equation*}
\sup_{x\in B_{\mathbb{R}^{N}}(1)}\left\Vert x\right\Vert _{L_{1}}=\sqrt{N}.
\end{equation*}%
By the definition of $f$ and (\ref{ex-new-condition}), it is direct to
verify that%
\begin{equation}
a=O(n)\text{, \ \ \ }b_{i}=O(\sum_{j}\left\vert A_{n}(i,j)\right\vert )+O(1)%
\text{, \ \ \ }c_{ij}=O(\left\vert A_{n}(i,j)\right\vert ).
\label{eq-new-ex-a-b-c}
\end{equation}%
Using (\ref{eq-new-ex-a-b-c}) it is straightforward to verify that the lower
bound part is implied by Theorem \ref{main-theorem}, that is,%
\begin{equation*}
\lim_{n\rightarrow \infty }\frac{1}{n}\left[ \log
\int_{W_{1}^{n}}e^{H_{n}^{J,h}(x)}d\mu (x)-\max_{\nu \ll \mu ,\nu =\nu
_{1}\times \nu _{2}\times \ldots \times \nu _{n},}\left\{ H_{n}^{J,h}(m(\nu
))-\sum_{i=1}^{n}D(\nu _{i}\parallel \mu _{i})\right\} \right] \geq 0.
\end{equation*}%
Next we consider the upper bound part. If we calculate $B_{1}$ and $B_{2}$
in Theorem \ref{main-theorem}, then they are of the wrong order. In order to
show Theorem \ref{theorem-as-new}, we need to incorporate the special
property of $f$ into the proof of Theorem \ref{main-theorem}. For $f$ and $%
\widetilde{f}$ we have that%
\begin{equation*}
f_{i}(x)(z)=\sum_{j\neq i}A_{n}(j,i)x_{j}^{T}Jz+h^{T}z\text{, \ \ \ }%
\widetilde{f}_{i}(x)(z)=\sum_{j\neq i}A_{n}(j,i)x_{j}^{T}Jz.
\end{equation*}%
Defining $\widetilde{\mu }$ same as (\ref{def-miu-hat}), we claim that%
\begin{equation}
\mathbb{E}_{\widetilde{\mu }}\left[ \left( \widetilde{f}(X)-\widetilde{f}(%
\widehat{X})\right) ^{2}\right] =o(n^{2}),  \label{eq-new-ex-dif-f-tilt}
\end{equation}%
and%
\begin{equation}
\mathbb{E}_{\widetilde{\mu }}\left[ \left( \sum_{i=1}^{n}\widetilde{f}%
_{i}(X)(X_{i}-\widehat{X}_{i})\right) ^{2}\right] =o(n^{2}).
\label{eq-new-ex-f-tilt-dif-x}
\end{equation}%
We defer their proofs to later place, and first show how to finish the proof
with them. By (\ref{eq-new-ex-dif-f-tilt}) and (\ref{eq-new-ex-f-tilt-dif-x}%
) we see that there exists $\sigma _{n}\rightarrow 0$ such that $\mathbb{P}_{%
\widetilde{\mu }}(\Omega _{n})\geq \frac{1}{2}$, where%
\begin{equation*}
\Omega _{n}:=\{x\in \text{supp}(\mu ):\left\vert \widetilde{f}(x)-\widetilde{%
f}(\widehat{x})\right\vert ,\left\vert \sum_{i=1}^{n}\widetilde{f}%
_{i}(x)(x_{i}-\widehat{x}_{i})\right\vert \leq \sigma _{n}n\}.
\end{equation*}%
Given any $\epsilon >0$, by \cite[Lemma 3.4]{AS} it is not hard to see that
we can construct a $\mathcal{D}(\epsilon )$ such that $\log \left\vert 
\mathcal{D}(\epsilon )\right\vert =o(n)$. For each $d\in \mathcal{D}%
(\epsilon )$, we consider%
\begin{equation*}
E_{d}:=\{x\in \text{supp}(\mu ):\sum_{i=1}^{n}\left\Vert
f_{i}(x)-d_{i}\right\Vert ^{2}\leq n\epsilon ^{2}\}\cap \Omega _{n}.
\end{equation*}%
If $E_{d}$ is not empty, we pick one element $z_{d}\in E_{d}$ and fix the
choice. Consider%
\begin{equation*}
\widetilde{\mathcal{D}}(\epsilon ):=\{z_{d}:d\in \mathcal{D}(\epsilon )\text{%
, }E_{d}\neq \varnothing \}.
\end{equation*}%
Then for any $x\in \Omega _{n}$, recalling the definition of $d^{x}$ (\ref%
{def_d}), we can find $y_{x}:=z_{d^{x}}\in \widetilde{\mathcal{D}}(\epsilon
) $, such that by the triangle inequality%
\begin{equation}
\sum_{i=1}^{n}\left\Vert f_{i}(x)-f_{i}(y_{x})\right\Vert ^{2}\leq
\sum_{i=1}^{n}\left\Vert f_{i}(x)-d_{i}^{x}\right\Vert
^{2}+\sum_{i=1}^{n}\left\Vert d_{i}^{x}-f_{i}(z_{d^{x}})\right\Vert ^{2}\leq
2n\epsilon ^{2}.  \label{eq-ex-new-covering-D-tilt}
\end{equation}%
Obviously $|\widetilde{\mathcal{D}}(\epsilon )|\leq \left\vert \mathcal{D}%
(\epsilon )\right\vert $ by the construction of $\widetilde{\mathcal{D}}%
(\epsilon )$, and thus%
\begin{equation}
\log \left\vert \widetilde{\mathcal{D}}(\epsilon )\right\vert =o(n).
\label{eq-new-ex-covering-condition}
\end{equation}%
Let $\widehat{y_{x}}=((\widehat{y_{x}})_{1},\ldots ,(\widehat{y_{x}})_{n})$,
where 
\begin{equation*}
\left( \widehat{y_{x}}\right) _{i}=\frac{\mathbb{E}_{\mu
_{i}}[x_{i}e^{\sum_{j\neq i}A_{n}(i,j)x_{i}^{T}J\left( y_{x}\right)
_{j}+x_{i}^{T}h}]}{\mathbb{E}_{\mu _{i}}[e^{\sum_{j\neq
i}A_{n}(i,j)x_{i}^{T}J\left( y_{x}\right) _{j}+x_{i}^{T}h}]}=\frac{\mathbb{E}%
_{\mu _{i}}[x_{i}e^{f_{i}(y_{x})(x_{i})}]}{\mathbb{E}_{\mu
_{i}}[e^{f_{i}(y_{x})(x_{i})}]}.
\end{equation*}%
Next we do the following approximation for $x\in \Omega _{n}$%
\begin{equation*}
f(x)\approx f(\widehat{y_{x}})-\sum_{i=1}^{n}D(\nu _{i}^{\widehat{y_{x}}%
}\parallel \mu _{i})+\sum_{i=1}^{n}\log \frac{d\nu _{i}^{\widehat{y_{x}}}}{%
d\mu _{i}}(x_{i}).
\end{equation*}%
More precisely, we show that for any $x\in \Omega _{n}$ 
\begin{equation}
\left\vert f(x)-\left( f(\widehat{y_{x}})-\sum_{i=1}^{n}D(\nu _{i}^{\widehat{%
y_{x}}}\parallel \mu _{i})+\sum_{i=1}^{n}\log \frac{d\nu _{i}^{\widehat{y_{x}%
}}}{d\mu _{i}}(x_{i})\right) \right\vert \leq 2\sigma _{n}n+2\sqrt{2N}%
n\epsilon .  \label{eq-new-ex-final-approximation}
\end{equation}%
If (\ref{eq-new-ex-final-approximation}) holds, then by (\ref%
{eq-new-ex-covering-condition}) and the same method as Section \ref%
{Section-upper-bound-final}, we see that {\small 
\begin{equation*}
\log \int_{W_{1}^{n}}e^{f(x)}d\mu (x)\leq \max_{\nu \ll \mu ,\nu =\nu
_{1}\times \nu _{2}\times \ldots \times \nu _{n}}\left\{ f(m(\nu
))-\sum_{i=1}^{n}D(\nu _{i}\parallel \mu _{i})\right\} +2\sigma _{n}n+2\sqrt{%
2N}n\epsilon +\log 2+\log \left\vert \mathcal{D}(\epsilon )\right\vert .
\end{equation*}%
} Dividing both sides by $n$, and noting the fact that $\epsilon $ is
arbitrary, we complete the proof by letting $\epsilon \rightarrow 0$.

Now we show (\ref{eq-new-ex-final-approximation}). 
%Note that by the form of $%
%f$ we have%
%\begin{equation*}
%\left( \widehat{y_{x}}\right) _{i}=\frac{\mathbb{E}_{\mu
%_{i}}[x_{i}e^{\sum_{j\neq i}A_{n}(i,j)x_{i}^{T}J\left( y_{x}\right)
%_{j}+x_{i}^{T}h}]}{\mathbb{E}_{\mu _{i}}[e^{\sum_{j\neq
%i}A_{n}(i,j)x_{i}^{T}J\left( y_{x}\right) _{j}+x_{i}^{T}h}]}=\frac{\mathbb{E}%
%_{\mu _{i}}[x_{i}e^{f_{i}(y_{x})(x_{i})}]}{\mathbb{E}_{\mu
%_{i}}[e^{f_{i}(y_{x})(x_{i})}]}.
%\end{equation*}%
Comparing the above equality with (\ref{v-exp-form}) in Proposition \ref%
{measure_v}, by (\ref{distribution_2}) and (\ref{distribution_4}) we see
that for any $z\in \mathbb{R}^{d}$%
\begin{equation*}
\log \frac{d\nu _{i}^{\widehat{y_{x}}}}{d\mu _{i}}(z)-D(\nu _{i}^{\widehat{%
y_{x}}}\parallel \mu _{i})=f_{i}(y_{x})(z-\left( \widehat{y_{x}}\right)
_{i}).
\end{equation*}%
Therefore we have{\small 
\begin{eqnarray}
&&\left\vert f(x)-\left( f(\widehat{y_{x}})-\sum_{i=1}^{n}D(\nu _{i}^{%
\widehat{y_{x}}}\parallel \mu _{i})+\sum_{i=1}^{n}\log \frac{d\nu _{i}^{%
\widehat{y_{x}}}}{d\mu _{i}}(x_{i})\right) \right\vert  \notag \\
&\leq &\left\vert \widetilde{f}(x)-\widetilde{f}(y_{x})\right\vert
+\left\vert \widetilde{f}(y_{x})-\widetilde{f}(\widehat{y_{x}})\right\vert
+\left\vert \sum_{i=1}^{n}\widetilde{f}_{i}(y_{x})(x_{i}-\left( y_{x}\right)
_{i})\right\vert +\left\vert \sum_{i=1}^{n}\widetilde{f}_{i}(y_{x})(\left(
y_{x}\right) _{i}-\left( \widehat{y_{x}}\right) _{i})\right\vert ,
\label{eq-new-ex-final-approximation-decompose}
\end{eqnarray}%
} where in the last line we replace $f$ by $\widetilde{f}$ since it is easy
to check that all the terms involving $h$ cancel in the first line.
Recalling that $y_{x}\in \Omega _{n}$, by the definition of $\Omega _{n}$ we
see that%
\begin{equation}
\left\vert \widetilde{f}(y_{x})-\widetilde{f}(\widehat{y_{x}})\right\vert
+\left\vert \sum_{i=1}^{n}\widetilde{f}_{i}(y_{x})(\left( y_{x}\right)
_{i}-\left( \widehat{y_{x}}\right) _{i})\right\vert \leq 2 \sigma_n n.
\label{eq-new-ex-final-approximation-eq-1}
\end{equation}%
Thus it remains to bound $\left\vert \widetilde{f}(x)-\widetilde{f}%
(y_{x})\right\vert $ and $\left\vert \sum_{i=1}^{n}\widetilde{f}%
_{i}(y_{x})(x_{i}-\left( y_{x}\right) _{i})\right\vert $. For $\left\vert 
\widetilde{f}(x)-\widetilde{f}(y_{x})\right\vert $, we have%
\begin{eqnarray}
\left\vert \widetilde{f}(x)-\widetilde{f}(y_{x})\right\vert &=&\left\vert 
\frac{1}{2}\sum_{i,j=1}^{n}A_{n}(i,j)\left( x_{i}^{T}J\left(
x_{j}-(y_{x})_{j}\right) +\left( x_{i}^{T}-(y_{x})_{i}^{T}\right)
J(y_{x})_{j}\right) \right\vert  \notag \\
&\leq &\left\vert \frac{1}{2}\sum_{i}\left( \widetilde{f}_{i}(x)-\widetilde{f%
}_{i}(y_{x})\right) \left( x_{i}\right) \right\vert +\left\vert \frac{1}{2}%
\sum_{j}\left( \widetilde{f}_{j}(x)-\widetilde{f}_{j}(y_{x})\right) \left(
y_{x}\right) \right\vert  \notag \\
&\leq &\sqrt{N}\sqrt{n}\left( \sum_{i}\left\Vert \widetilde{f}_{i}(x)-%
\widetilde{f}_{i}(y_{x})\right\Vert ^{2}\right) ^{1/2}\leq \sqrt{2N}%
n\epsilon ,  \label{eq-new-ex-diff-f-x-f-y-x}
\end{eqnarray}%
where the last inequality is by (\ref{eq-ex-new-covering-D-tilt}). For $%
\left\vert \sum_{i=1}^{n}\widetilde{f}_{i}(y_{x})(x_{i}-\left( y_{x}\right)
_{i})\right\vert $, note that%
\begin{equation*}
\left\vert \sum_{i=1}^{n}\widetilde{f}_{i}(y_{x})(x_{i}-\left( y_{x}\right)
_{i})\right\vert =\left\vert \sum_{i,j=1}^{n}A_{n}(i,j)\left(
x_{i}^{T}-(y_{x})_{i}^{T}\right) J(y_{x})_{j}\right\vert =\left\vert
\sum_{j}\left( \widetilde{f}_{j}(x)-\widetilde{f}_{j}(y_{x})\right) \left(
y_{x}\right) \right\vert ,
\end{equation*}%
which we already bound in (\ref{eq-new-ex-diff-f-x-f-y-x}). Thus combining (%
\ref{eq-new-ex-final-approximation-decompose}), (\ref%
{eq-new-ex-final-approximation-eq-1}) and (\ref{eq-new-ex-diff-f-x-f-y-x}),
we get (\ref{eq-new-ex-final-approximation}).

In the following we prove (\ref{eq-new-ex-dif-f-tilt}) and (\ref%
{eq-new-ex-f-tilt-dif-x}). We need the following two inequalities. By (\ref%
{ex-new-condition}) and (\ref{eq-new-ex-a-b-c}), there exists $\eta _{n}=o(n)
$, such that for any $w_{1},w_{2},\ldots ,w_{n}\in W_{1}^{n}$%
\begin{equation}
\sum_{i=1}^{n}\left\Vert \widetilde{f_{i}}(w_{i})\right\Vert ^{2}\leq
C(n+\sum_{i=1}^{n}(\sum_{j=1}^{n}\left\vert A_{n}(i,j)\right\vert )^{2})\leq
C(n+n\sum_{i,j=1}^{n}\left\vert A_{n}(i,j)\right\vert ^{2})\leq
C(n+ntr\left( A_{n}^{2}\right) )=\eta _{n}n^{2},  \label{eq-new-ex-sig-f-i-2}
\end{equation}%
and by (\ref{ex-new-condition}) again, there exists $M_{n}=O(1)$ such that
for all $x\in W_{1}^{n}$ 
\begin{equation}
\sum_{i=1}^{n}\left\Vert \widetilde{f_{i}}(x)\right\Vert \leq
\sum_{i=1}^{n}\left\Vert \sum_{j\neq i}A_{n}(i,j)Jx_{j}\right\Vert _{\infty
}\leq \sqrt{N}\left\Vert J\right\Vert _{\infty }\sup_{x\in \lbrack
0,1]^{n}}\sum_{i\in \lbrack n]}\left\vert \sum_{j\in \lbrack
n]}A_{n}(i,j)x_{j}\right\vert \leq M_{n}n.  \label{eq-new-ex-sig-f-i}
\end{equation}

\begin{proof}[Proof of (\protect\ref{eq-new-ex-dif-f-tilt})]
If we directly apply Proposition \ref{Prop-diff-g} for $\widetilde{f} $,
then we can see that only $\sum_{i,j=1}^{n}b_{i}b_{j}c_{ij}$ is of wrong
order, which comes from (\ref{upper_6}). Let $\theta =(0,0,\ldots ,0)$ in $%
\mathbb{R}^{N}$. Here we show 
\begin{equation}
\sum_{i=1}^{n}\mathbb{E}_{\widetilde{\mu }}\left[ u_{i}(t,X)(X_{i}-\widehat{X%
}_{i})\left( h(X)-h(X_{\theta _{i}}^{(i)})\right) \right] = o(n^2),
\label{eq-4.40-asist}
\end{equation}
which gives (\ref{eq-new-ex-dif-f-tilt}). Recall that%
\begin{equation*}
h(x)=\widetilde{f}(x)-\widetilde{f}(\widehat{x}),\text{ }u_{i}(t,x)=%
\widetilde{f}_{i}(tx+(1-t)\widehat{x}).
\end{equation*}%
By arrangement we have{\small 
\begin{eqnarray}
h(x)-h(x_{\theta _{i}}^{(i)}) &=&\widetilde{f}(x)-\widetilde{f}(x_{\theta
_{i}}^{(i)})+\frac{1}{2}\sum_{l,j}A_{n}(l,j)(\widehat{x}_{l})^{T}J\left( 
\widehat{x}_{j}-\widehat{x_{\theta _{i}}^{(i)}}_{j}\right) +\frac{1}{2}%
\sum_{l,j}A_{n}(l,j)(\widehat{x}_{l}-\widehat{x_{\theta _{i}}^{(i)}}%
_{l})^{T}J\widehat{x_{\theta _{i}}^{(i)}}_{j}  \notag \\
&=&\widetilde{f}(x)-\widetilde{f}(x_{\theta _{i}}^{(i)})+\frac{1}{2}\sum_{j}%
\widetilde{f}_{j}(\widehat{x})\left( \widehat{x}_{j}-\widehat{x_{\theta
_{i}}^{(i)}}_{j}\right) +\frac{1}{2}\sum_{l}\widetilde{f}_{l}(\widehat{%
x_{\theta _{i}}^{(i)}})\left( \widehat{x}_{l}-\widehat{x_{\theta _{i}}^{(i)}}%
_{l}\right) .  \label{eq-new-ex-decompose-h}
\end{eqnarray}%
%
%x
} We also have{\small 
\begin{equation}
\widetilde{f}(x)-\widetilde{f}(x_{\theta _{i}}^{(i)})=\frac{1}{2}%
\sum_{l,j}A_{n}(l,j)(x_{l})^{T}Jx_{j}-\frac{1}{2}\sum_{l,j}A_{n}(l,j)((x_{%
\theta _{i}}^{(i)})_{l})^{T}Jx_{j}^{(i)}=\sum_{j\neq
i}A_{n}(i,j)(x_{i})^{T}Jx_{j}=\widetilde{f}_{i}(x)(x_{i}).
\label{eq-new-ex-dif-f-tilt-x}
\end{equation}
} By Cauchy inequality we have%
\begin{eqnarray*}
&&\left\vert \sum_{i=1}^{n}\mathbb{E}\left[ \left( \left\Vert \widetilde{f}%
_{i}(X)\right\Vert +\left\Vert \widetilde{f}_{i}(\widehat{X})\right\Vert
\right) \left( \left\Vert \widetilde{f}_{i}(X)\right\Vert \right) \right]
\right\vert \\
&\leq &\left(\mathbb{E} \sum_{i=1}^{n}\left\Vert \widetilde{f}%
_{i}(X)\right\Vert ^{2}\right)+\left( \mathbb{E} \sum_{i=1}^{n}\left\Vert 
\widetilde{f}_{i}(\widehat{X})\right\Vert ^{2}\right) ^{1/2}\left(\mathbb{E}
\sum_{i=1}^{n}\left\Vert \widetilde{f}_{i}(X)\right\Vert ^{2}\right) ^{1/2}
= o( n^{2}),
\end{eqnarray*}%
where the last line is by (\ref{eq-new-ex-sig-f-i-2}). Thus with (\ref%
{eq-new-ex-dif-f-tilt-x}) we see that%
\begin{eqnarray}
&&\left\vert \sum_{i=1}^{n}\mathbb{E}\left[ u_{i}(t,X)(X_{i}-\widehat{X}%
_{i})\left( \widetilde{f}(X)-\widetilde{f}(X^{(i)})\right) \right]
\right\vert  \notag \\
&=&\left\vert \sum_{i=1}^{n}\mathbb{E}\left[ \left( t\widetilde{f}%
_{i}(X)+(1-t)\widetilde{f}_{i}(\widehat{X})\right) (X_{i}-\widehat{X}%
_{i})\left( \widetilde{f}_{i}(X)(X_{i})\right) \right] \right\vert  \notag \\
&\leq &N\left\vert \sum_{i=1}^{n}\mathbb{E}\left[ \left( \left\Vert 
\widetilde{f}_{i}(X)\right\Vert +\left\Vert \widetilde{f}_{i}(\widehat{X}%
)\right\Vert \right) \left( \left\Vert \widetilde{f}_{i}(X)\right\Vert
\right) \right] \right\vert =o(n^{2}).  \label{eq-ex-new-add-319-2}
\end{eqnarray}%
Next we define $\Delta _{j,i}(x):=\widehat{x}_{j}-\widehat{x^{(i)}}_{j}$,
and then by (\ref{bound-derivative-x-hat}), (\ref{eq-new-extra-condition})
and (\ref{eq-new-ex-a-b-c}) we see that for any $x \in W_1^n$,%
\begin{equation}
\max_{i,j}\left\vert \Delta _{j,i}(x)\right\vert \leq \sqrt{N}%
\max_{i,j}\left\vert A_{n}(i,j)\right\vert =o(1)\text{.}
\label{eq-new-ex-delta-bound}
\end{equation}%
By (\ref{eq-new-ex-sig-f-i}), for any $x \in W_1^n$,%
\begin{equation}
\sum_{i,j=1}^{n}\left\Vert \widetilde{f}_{j}(\widehat{x})\right\Vert \left(
\left\Vert \widetilde{f}_{i}(x)\right\Vert +\left\Vert \widetilde{f}_{i}(%
\widehat{x})\right\Vert \right) \leq \left( \sum_{i=1}^{n}\left\Vert 
\widetilde{f}_{j}(\widehat{x})\right\Vert \right) \left(
\sum_{i=1}^{n}\left\Vert \widetilde{f}_{i}(x)\right\Vert
+\sum_{i=1}^{n}\left\Vert \widetilde{f}_{i}(\widehat{x})\right\Vert \right)
\leq 2 M_n^2 n^{2}.  \label{eq-ex-new-add-319-1}
\end{equation}%
Noting that $|\widetilde{f}_{j}(\widehat{x})(\Delta
_{j,i}(x)^{T})(u_{i}(t,x)(x_{i}-\widehat{x}_{i}))|\leq ||\widetilde{f}_{j}(%
\widehat{x})||(||\widetilde{f}_{i}(x)||+||\widetilde{f}_{i}(\widehat{x})||)$%
, with (\ref{eq-new-ex-delta-bound}) and (\ref{eq-ex-new-add-319-1}) we see
that%
\begin{eqnarray}
&&\left\vert \sum_{i=1}^{n}\mathbb{E}\left[ u_{i}(t,X)(X_{i}-\widehat{X}%
_{i})\left( \frac{1}{2}\sum_{j}\widetilde{f}_{j}(\widehat{X})\left( \widehat{%
X}_{j}-\widehat{X^{(i)}}_{j}\right) \right) \right] \right\vert  \notag \\
&=&\left\vert \frac{1}{2}\mathbb{E}\left[ \sum_{i,j=1}^{n}\widetilde{f}_{j}(%
\widehat{X})\left( \Delta _{j,i}(X)^{T}\right) \left( u_{i}(t,X)(X_{i}-%
\widehat{X}_{i})\right) \right] \right\vert =o(n^{2}).
\label{eq-new-ex-extra-lemma-1}
\end{eqnarray}%
Similarly we can show that%
\begin{equation}
\left\vert \sum_{i=1}^{n}\mathbb{E}\left[ u_{i}(t,X)(X_{i}-\widehat{X}%
_{i})\left( \frac{1}{2}\sum_{l}\widetilde{f}_{j}(\widehat{X^{(i)}})\left( 
\widehat{X}_{l}-\widehat{X^{(i)}}_{l}\right) \right) \right] \right\vert
=o(n^{2}),  \label{eq-new-ex-extra-lemma-3}
\end{equation}%
and thus with (\ref{eq-ex-new-add-319-2}), (\ref{eq-new-ex-extra-lemma-1}), (%
\ref{eq-new-ex-extra-lemma-3}) and (\ref{eq-new-ex-decompose-h}), we get (%
\ref{eq-4.40-asist}) and finish the proof.
\end{proof}

\begin{proof}[Proof of (\protect\ref{eq-new-ex-f-tilt-dif-x})]
Denote%
\begin{equation*}
G(x):=\sum_{i=1}^{n}\widetilde{f}_{i}(x)(x_{i}-\widehat{x}_{i}).
\end{equation*}%
Then by the definition of $\widehat{X}_{i}$ we have%
\begin{equation*}
\mathbb{E}_{\widetilde{\mu }}\left[ G(X^{(i)})\widetilde{f}%
_{i}(X^{(i)})(X_{i}-\widehat{X}_{i})\right] =0.
\end{equation*}%
Noting that $\widetilde{f}_{i}(X^{(i)})=\widetilde{f}_{i}(X)$, it is enough
to show that%
\begin{equation}
\mathbb{E}_{\widetilde{\mu }}\left[ \sum_{i=1}^{n}\widetilde{f}%
_{i}(X^{(i)})(X_{i}-\widehat{X}_{i})\left( G(X)-G(X^{(i)})\right) \right]
=o(n^{2}).  \label{eq-ex-new-add-319-3}
\end{equation}%
After some algebra we have%
\begin{eqnarray}
G(X)-G(X^{(i)}) &=&\sum_{j=1}^{n}\sum_{l=1}^{n}A_{n}(l,j)\left(
X_{l}^{T}J(X_{j}-\widehat{X}_{j})-\left( X_{l}^{(i)}\right)
^{T}J(X_{j}^{(i)}-\widehat{X^{(i)}}_{j})\right)  \notag \\
&=&\sum_{j\neq i}^{n}\widetilde{f}_{j}(X)\left( \widehat{X^{(i)}}_{j}-%
\widehat{X}_{j}\right) +\left( 2\widetilde{f}_{i}(X)-\widetilde{f}_{i}(%
\widehat{X^{(i)}})\right) (X_{i}).  \label{eq-new-ex-decompose-dif-g}
\end{eqnarray}%
By the definition of $\Delta _{j,i}$, we have%
\begin{eqnarray}
&&\mathbb{E}_{\widetilde{\mu }}\left[ \left\vert \sum_{i=1}^{n}\widetilde{f}%
_{i}(X^{(i)})(X_{i}-\widehat{X}_{i})\left( \sum_{j\neq i}^{n}\widetilde{f}%
_{j}(X)\left( \widehat{X}_{j}-\widehat{X^{(i)}}_{j}\right) \right)
\right\vert \right]  \notag \\
&\leq & \mathbb{E}_{\widetilde{\mu }} \left[ \max_{i,j}\left\Vert \Delta
_{j,i}(X)\right\Vert \cdot \sqrt{N}\left( \sum_{i=1}^{n}\left\Vert 
\widetilde{f}_{i}(X^{(i)})\right\Vert \right) \left(
\sum_{j=1}^{n}\left\Vert \widetilde{f}_{j}(X)\right\Vert \right) \right]
=o(n^{2}).  \label{eq-new-ex-extra-lemma-4}
\end{eqnarray}%
Also we have{\small 
\begin{align}
\mathbb{E}_{\widetilde{\mu }}\left[ \left\vert \sum_{i=1}^{n}\widetilde{f}%
_{i}(X^{(i)})(X_{i}-\widehat{X}_{i})\left( \left( 2\widetilde{f}_{i}(X)-%
\widetilde{f}_{i}(\widehat{X^{(i)}})\right) (X_{i})\right) \right\vert %
\right] \\
\leq \mathbb{E}_{\widetilde{\mu }}\left[N\sum_{i=1}^{n}\left\Vert \widetilde{%
f}_{i}(X^{(i)})\right\Vert \left( 2\left\Vert \widetilde{f}%
_{i}(X)\right\Vert +\left\Vert \widetilde{f}_{i}(\widehat{X})\right\Vert
\right)\right] =o(n^{2}),  \label{eq-new-ex-extra-lemma-5}
\end{align}%
} where the last line is by Cauchy inequality and (\ref{eq-new-ex-sig-f-i-2}%
). Combining (\ref{eq-ex-new-add-319-3}), (\ref{eq-new-ex-decompose-dif-g}),
(\ref{eq-new-ex-extra-lemma-4}) and (\ref{eq-new-ex-extra-lemma-5}), we
finish the proof.
\end{proof}

\section*{Acknowledgements}

The author thank Sourav Chatterjee for proposing this research direction.
The author also thank Sourav Chatterjee, Amir Dembo, Sumit Mukherjee and
Yufei Zhao for helpful discussions. The author is grateful to two anonymous
referees for their detailed comments and suggestions, which have improved
the quality of this paper. %\bibliographystyle{plain}
\bibliographystyle{plain}
\bibliography{GNLD}

\end{document}